\documentclass[12pt]{article}
\usepackage{ayv_paperstyle}
\newcommand{\ptd}{\mathcal{P}^2(\td)}
\usepackage{mathbbol}
\usepackage{bbm}

\title{Lattice approximations of the first-order mean field type differential games}
\author{Yurii Averboukh\footnote{Krasovskii Institute of Mathematics and Mechanics\ \ \email{ayv@imm.uran.ru}}{ }\footnote{Ural Federal University}}
\date{}

\begin{document}
	\maketitle
\begin{abstract}
The theory of first-order mean field type differential games examines the systems of infinitely many identical agents interacting via some external media under assumption that each agent is controlled by two players. We study the approximations of the value function of the first-order mean field type differential game using solutions of model finite-dimensional differential games.
The model game appears as a mean field type continuous time Markov game, i.e., the game theoretical problem with the infinitely many agents and dynamics of each agent determined by a controlled finite state nonlinear Markov chain. Given a supersolution (resp. subsolution) of the Hamilton-Jacobi equation for the model game, we construct a suboptimal strategy of the first (resp. second) player and evaluate the approximation accuracy using the modulus of continuity of the reward function and the distance between the original and model games. This gives the approximations of the value function of the mean field type differential game by  values  of the finite-dimensional differential games. Furthermore, we present the way to build a finite-dimensional differential game that approximates the original game with a given accuracy.
\msccode{49L25, 49J45, 49J53, 49N70}
\end{abstract}
\section{Introduction}
The mean field type differential games are mathematical models of systems consisting of many agents those are influenced by two player. The main feature of the mean field approach is the assumption that the agents interact via some external media, i.e., the dynamics of each agent depends only on time, his/her state,  controls produced by the players and the distribution of all agents. In the paper, we study the zero-sum first-order mean field type differential games. This means that the agents' dynamics is given by the ordinary differential equation, when the players' purposes are opposite. 

Notice that the theory of mean field type differential games comprises the mean field type control problems. In fact, one can regard the  mean field type control problem  as the mean field type differential games with the fictions second player. It is worth to mention that the mean filed type control theory is a counterpart of the mean field game theory introduced by Lasry and Lions~\cite{Lions01},~\cite{Lions02} and (independently) by Huang, Malham\'{e},  Caines~\cite{Huang5}. However, the mean field game theory implies that each agent is a player and tries to optimize his/her own payoff.

The mean field type  control theory started with  paper~\cite{ahmed_ding_controlld}. Nowadays, for the mean field type control system  such methods as dynamic programming and analogs of Pontryagin maxim principle are developed~\cite{Andersson_Djehiche_2011},~\cite{Bayraktar_Cosso_Pham_randomized},~\cite{Bensoussan_Frehse_Yam_book},~\cite{Buckdahn_Boualem_Li_PMP_SDE},~\cite{Carmona_Delarue_PMP},~\cite{Lauriere_Pironneau_DPP_MF_control},~\cite{Pham_Wei_2015_DPP_2016},~\cite{Pham_Wei_2015_Bellman}. Notice that the first-order mean field type control problems are more challenging due to the possible crowding of the agents. They were examined from the dynamic programming point of view in~\cite{Marigonda_Cavagnari},~\cite{Marigonda_et_al_2015},~\cite{Jimenez_Marigonda_Quincampoix}. 

First, the mean field type differential games was analyzed in~\cite{Djehiche_Hamdine}. In that the paper it was assumed that the dynamics of each agent is given by a  non-degenerate stochastic differential equation. The main result is the existence of the value function in the class of progressively measurable strategies. Papers ~\cite{Cosso_Pham},~\cite{Moon_Basar} examine the mean field type differential games using nonaticipative strategies. There is proved the existence  theorem for the value function. The feedback approach to the mean field type differential games is developed in~\cite{Averbokh_mfdg}. That paper extends the Krasovskii-Subbotin extremal shift rule   to the mean-field type control system. Furthermore, the existence of the value function is obtained using so called programmed iteration method. The infinitesimal characterization of the value function is derived in~\cite{Averboukh_dirderivative} and~\cite{Cosso_Pham}. Paper~\cite{Cosso_Pham} shows that the value function of the mean field type differential game is the viscosity solution of the corresponding Hamilton-Jacobi equation in the space of probability measures. The notion of viscosity solution used there is based on embedding of the set of probability measures into the space of square integrable random variables. The characterization of the value function of the mean field type differential game via minimax inequalities on  direction derivatives was obtained in~\cite{Averboukh_dirderivative}.

The aforementioned  results are counterparts of the finite-dimensional differential games theory. However, the mean field type control theory raises its own questions. One of these questions is a finite dimensional approximations. Recall that the mean field type control theory implies that the phase space of the whole system is  the Wasserstein space that is only metric, when the corresponding Hamilton-Jacobi equation involves the derivatives with respect to probability (various approaches to this problem can be found in \cite{Cosso_Pham}, \cite{Marigonda_Quincampoix}, \cite{Pham_Wei_2015_DPP_2016}, \cite{Pham_Wei_2015_Bellman}). 

At the same time, the  finite-dimensional differential games are  simpler and well-studied~\cite{bardi},\cite{NN_PDG_en},~\cite{Subb_book}. In particular,  they are  reduced to the  finite dimensional Hamilton-Jacobi equations. Thus, it is tempting to approximate a mean field type control system by finite dimensional systems and construct finite-dimensional approximations of the value function. Since the mean field type control system implies the infinitely many agents, the most natural approximation arises if we consider finite agent system assuming that the probability in the right-hand side of the dynamics is determined by the empirical measure.  This approach was realized in~\cite{Gangbo_Mayorga_Swiek},~\cite{Lacker_McKean_Vlasov}.  There, it is shown that solutions of the control problems with finite  number of agents converge to the solution of the limiting mean field control problem when the number of agents tends to infinity. This result, in particular, shows that the mean field type control theory is an appropriate model for systems with large number of agents. However, the aforementioned papers do not provide the approximation rate.

We develop a quite different approach. We examine a mean field type differential game assuming, for simplicity, that the phase space of each agent is the flat torus and approximate it  by the finite-state mean field type Markov game, i.e.,  the approximative system consists now of infinitely many identical agents with dynamics of each agent given by the controlled continuous-time Markov chain defined on some lattice.  It is well known that the dynamics of the distribution of agents in the Markov chain obeys the Kolmogorov equation which, for the mean field type case, is a system of nonlinear ODEs. Thus, the approximative problem turns to be a finite-dimensional differential game. Notice that, if the agents' state space is  Euclidean, such approximation can be performed using a lattice contained in a sufficiently large ball. 

The main result of the paper (see Theorem~\ref{th:existence_strategy}) states that, if the dynamics of each agent according to the Markov chain approximates the original deterministic evolution, then the value function of the finite-state  mean field type Markov game provides the approximation of the value function of the original mean field type differential game. Moreover, the approximation accuracy can be estimated by the modulus of continuity of the reward function and distance between the original and approximative systems. The proof is constructive, i.e., given a supersolution (respectively, subsolution) of the Hamilton-Jacobi equation corresponding to the mean field type Markov game, we construct a suboptimal strategy of the first (respectively, second) player in the original mean field type differential game providing the desired approximation accuracy. Furthermore,  one can always build at least one approximative system defined on a regular lattice (see Section~\ref{sect:example_construction}).

The paper is organized as follows. In Section~\ref{sect:preliminaries}, we introduce the general notation used throughout the paper (see Subsection~\ref{sub:sect:notation}) and recall  the feedback approach to the mean field type differential games first developed in \cite{Averbokh_mfdg} (see Subsection~\ref{sub:sect:MFTDG}). Section~\ref{sect:finite_state_games} is concerned with the finite-state mean field type Markov games. Here, we introduce this object and, utilizing the relaxed feedback strategies, reduce such  game theoretical problems to a finite-dimensional differential games defined on a simplex. The main result of the paper is formulated in Section~\ref{sect:result}. Recall that it states that, given a mean filed type Markov game that approximates the original one and the supersolution (respectively, subsolution) of corresponding finite-dimensional Hamilton-Jacobi equation, one can find a subsoptimal strategy of the first (respectively, second) player in the original game. A construction of the mean-field type Markov game that approximates the original game with given accuracy  is presented in Section~\ref{sect:example_construction}. The rest of the paper is devoted to the proof of the main result. In Section~\ref{sect:properties}, we give the properties of  the original and model systems those are utilized in the proof of the main theorem. We introduce the desired subotimal strategy with memory in Section~\ref{sect:construction_suboptimal}. Its construction relies on the extremal shift rule first proposed by Krasovskii and Subbotin for finite-dimensional differential games~\cite{NN_PDG_en}.  Finally, proof of the main the theorem is given in Section~\ref{sect:extremal_shift}.

\section{Preliminaries}\label{sect:preliminaries}

\subsection{General notation}\label{sub:sect:notation}
\begin{itemize}
	\item If $X_1,\ldots,X_n$ are the set, $i_1,\ldots,i_k$ are distinct indexes, then
	$\operatorname{p}^{i_1,\ldots,i_k}:X_1\times \ldots\times X_{n}\rightarrow X_{i_1}\times \ldots\times X_{i_k}$ acts by the rule
	$$\operatorname{p}^{i_1,\ldots,i_k}(x_1,\ldots,x_n)\triangleq (x_{i_1},\ldots,x_{i_k}). $$
	\item Let $(\Omega',\mathcal{F}')$ and  $(\Omega'',\mathcal{F}'')$ be measurable spaces, $m$ be a measure on $\mathcal{F}'$, $h:\Omega'\rightarrow\Omega''$ be a measurable function. Denote by $h\sharp m$ the push-forward measure defined by the rule: for $\Upsilon\in \mathcal{F}''$,
	$$(h\sharp m)(\Upsilon)\triangleq m(h^{-1}(\Upsilon)). $$
	\item If  $(X,\rho_X)$, $(Y,\rho_Y)$ are metric spaces, then $C(X,Y)$ denotes the set of continuous functions from $X$ to $Y$. We write $C(X)$ for $C(X,\mathbb{R})$.
	\item $C_b(X)$ stands for the set of bounded continuous functions from $X$ to $\mathbb{R}$.
	\item Denote by $\mathcal{M}^+(X)$ the set of (nonnegative) measures on $X$. We endow $\mathcal{M}^+(X)$ with the topology of narrow convergence: a sequence $\{m_n\}_{n=1}^\infty\subset\mathcal{M}^+(X)$ narrowly converges to $m\in\mathcal{M}^+(X)$ if, for any $\phi\in C_b(X)$, \[\int_X\phi(x)m_n(dx)\rightarrow\int_X\phi(x)m(dx)\text{ as }n\rightarrow\infty.\].
	\item  $\mathcal{P}(X)$ is the set of probability measures, i.e.,
	$$\mathcal{P}(X)\triangleq \{m\in\mathcal{M}^+(X):m(X)=1\}. $$
	\item We denote the Dirac measure concentrated at $z\in X$ by $\delta_z$, i.e., for a Borel set $\Upsilon$,
	$$\delta_z(\Upsilon)\triangleq\left\{\begin{array}{cc}
	1, & z\in \Upsilon, \\
	0, & z\notin \Upsilon.
	\end{array}\right. $$
	\item For $p\geq 1$, let $\mathcal{P}^p(X)$ stand for the set of probability measures with finite $p$-th moment i.e.
	$m\in\mathcal{P}(X)$ lies in  $\mathcal{P}^p(X)$, if, for some (equivalently, every) $x_0\in X$,
	$$\int_X(\rho_X(x,x_0))^pm(dx)<\infty. $$
	\item We consider on $\mathcal{P}^p(X)$ the $p$-Wasserstein distance defined as follows:
	$$W_p(m_1,m_2)\triangleq \left[\inf_{\pi\in\Pi(m_1,m_2)}\int_{X\times X} (\rho_X(x_1,x_2))^p\pi(d(x_1,x_2))\right]^{1/p}. $$ Here $\Pi(m_1,m_2)$ is the set of probabilities on $X\times X$ such that $\operatorname{p}^i\sharp\pi=m_i$, $i=1,2$. Notice that $\mathcal{P}^p(X)$ endowed with the metric $W_p$ is Polish \cite{Ambrosio}. If, additionally, $X$ is compact, then $\mathcal{P}^p(X)$ is also compact \cite{Ambrosio}.
	In the paper we primary consider the quadratic case, i.e., $p=2$.
	\item Given $m\in\mathcal{M}^+(X)$, let $\mathcal{M}^+_{\text{product}}(X,m,Y)$ be a set of measures on $X\times Y$ compatible with $m$, i.e., $\alpha\in\mathcal{M}^+_{\text{product}}(X,m,Y)$ if $\alpha\in\mathcal{M}^+(X\times Y)$ and $\operatorname{p}^1\sharp\alpha=m$.
	The set $\mathcal{M}^+_{\text{product}}(X,m,Y)$ inherits the topology of narrow convergence from $\mathcal{M}^+(X\times Y)$. Notice that $\mathcal{M}^+_{\text{product}}(X,m,Y)$ is compact provided that both $X$ and $Y$ are compact. 
	\item The function $\gamma:X\rightarrow \mathcal{P}(Y)$ is called weakly measurable if, for every $\phi\in C_b(X\times Y)$, the function $x\mapsto \int_Y\phi(x,y)\gamma(x,dy)$ is measurable. If $m$ is a measure on $X$, $\gamma:X\rightarrow \mathcal{P}(X)$ is weakly measurable, then define the measure $m\star\gamma$ by the rule: for $\varphi\in C_b(X\times Y)$,
	\begin{equation*}\label{equality:m_star_xi}\int_{X\times Y}\phi(x,y)(m\star\gamma)(d(x,y))=
	\int_{X}\int_Y\phi(x,y)\gamma(x,dy)m(dx).
	\end{equation*} Obviously, $m\star\gamma\in\mathcal{M}^+_\text{product}(X,m,Y)$.
	\item  The disintegration theorem \cite{Ambrosio} states that, given $\alpha\in\mathcal{M}^+_{\text{product}}(X,m,Y)$, one can construct a weakly measurable function $\gamma:X\rightarrow\mathcal{P}(Y)$ such that $\alpha=m\star\gamma$. If $\gamma$ is a disintegration of $\alpha$, we denote \[\alpha(dy|x)\triangleq \gamma(x,dy).\] The disintegration is unique $m$-a.e. 
\end{itemize}

\subsection{Mean field type differential game}\label{sub:sect:MFTDG}
The main object of the paper is the  infinite system of identical agents with the dynamics governed by ordinary differential equations influenced by two players with opposite purposes. We assume the mean field interaction between the agents.  For simplicity, let the   phase space of each agent be $\td\triangleq\rd/\mathbb{Z}^d$. An element of $\td$ is a  set $x=\{x'+n:n\in\mathbb{Z}^d\}$. We denote the metric on $\td$ by $\|x-y\|$: \[\|x-y\|=\min\{\|x'-y'\|:x'\in x,\ \ y'\in y\}.\] 


Let $\mathcal{C}_{s,r}\triangleq C([s,r],\td)$ be the set of all trajectories on $[s,r]$. If $x(\cdot),y(\cdot)\in \mathcal{C}_{s,r}$, then distance between $x(\cdot),y(\cdot)$ is denoted (with some abuse of notation) by $\|x(\cdot)-y(\cdot)\|$
$$\|x(\cdot)-y(\cdot)\|\triangleq \sup_{t\in [s,r]}\|x(t)-y(t)\|.$$ Furthermore, if $t\in [s,r]$, then denote by $e_t$ the evaluation operator from $\mathcal{C}_{s,r}$ to $\td$ acting by the rule
$$e_t(x(\cdot))\triangleq x(t). $$

Let the dynamics of each agent be given by
\begin{equation}\label{sys:agent_mfdg}
\begin{split}
\frac{d}{dt}x(t)=f(&t,x(t),m(t),u(t),v(t)),\\ &t\in [0,T],\ \ x(t)\in \td,\ \ m(t)\in\mathcal{P}^2(\td),\ \ u(t)\in U,\ \ v(t)\in V.
\end{split}
\end{equation} 

It is assumed that $u(t)$ and $v(t)$ are control of the first and second players respectively. The concept of mean filed type control processes implies that agents are influenced by the players independently. 

We assume that the first (respectively, second) player tries to minimize (respectively, maximize) the value
$$g(m(T)). $$ The following conditions are imposed:
\begin{itemize}
	\item the sets $U$ and $V$ are metric compacts;
	\item the functions $f$ and $g$ are continuous;
	\item the function $f$ is Lipschitz continuous with respect to $x$ and $m$, i.e., there exists a constant $L>0$ such that, for any $t\in [0,T]$, $x_1,x_2\in\td$, $m_1,m_2\in\ptd$, $u\in U$, $v\in V$,
	\begin{equation}\label{intro:L}
	\|f(t,x_1,m_1,u,v)-f(t,x_2,m_2,u,v)\|\leq L(\|x_1-x_2\|+W_2(m_1,m_2)); 
	\end{equation}
	\item (Isaacs' condition) for every $t\in [0,T]$, $x\in \td$, $m\in\ptd$, and $w\in\rd$,
	$$\min_{u\in U}\max_{v\in V}\langle w,f(t,x,m,u,v)\rangle=\max_{v\in V}\min_{u\in U}\langle w,f(t,x,m,u,v)\rangle .$$
\end{itemize}

These conditions implies that 
\begin{equation}\label{intro:R}
R=\max_{x\in\td}\|f(t,x,m,u,v\|<\infty.
\end{equation}

Let us recall the definitions of open-loop and feedback strategies and corresponding motions of system (\ref{sys:agent_mfdg})  proposed in~\cite{Averbokh_mfdg}.  First, we consider the motion of each agent assuming that the flow of probabilities $m(\cdot)$ is given. Then, we introduce the flow of probabilities generated by the distribution of open-loop strategies. Finally, we introduce the feedback strategies and strategies with memory of the players and the corresponding flows of probabilities. 

As it was mentioned above, we start with the motions of each player. First, recall the classes of strategies.
\begin{itemize}
	\item A constant control of the first (respectively, second) player is an element  $u\in U$ (respectively, $v\in V$).
	\item A measurable control of the first player is a measurable function $u:[0,T]\rightarrow U$. Analogously, a second player's measurable control is a measurable function $v:[0,T]\rightarrow V$. Denote the set of measurable controls of the first (respectively, second) player by $\mathcal{U}_0$ (respectively, $\mathcal{V}_0$).
	\item A relaxed control of the first (respectively, second) player is a measure on $[0,T]\times U$ (respectively, on $[0,T]\times V$) compatible with the Lebesgue measure. This means that the set of relaxed controls of the first player is   \[\mathcal{U}\triangleq \mathcal{M}^+_{\text{product}}([0,T],\lambda,U),\] where $\lambda$ stands for the Lebesgue measure on $[0,T]$. Analogously, the set of relaxed control of the second player is  \[\mathcal{V}\triangleq \mathcal{M}^+_{\text{product}}([0,T],\lambda,V).\] This means that, given $\xi\in\mathcal{U}$, $\zeta\in\mathcal{V}$, we regard their disintegrations with respect to the Lebesgue measure $\xi(\cdot|t)$ and $\zeta(\cdot|t)$ as distributions of instantaneous players' controls applied at time $t$.  
\end{itemize} 

Notice that the sets of constant control are embedded into the sets of measurable controls. Furthermore, we identify a measurable control $u(\cdot):[0,T]\rightarrow U$ with a relaxed control $t\mapsto \delta_{u(t)}$. Recall that $\delta_z$ denotes the Dirac measure concentrated at $z$. One can also identify a measurable function $v(\cdot)$ with a weakly measurable mapping $t\mapsto \delta_{v(t)}$. This provides the embedding of the sets of measurable controls of the players to the set of relaxed controls.  Below, with some abuse of notation we assume that
\begin{equation}\label{inclusion:U_U_relaxed}
U\subset\mathcal{U}_0\subset\mathcal{U}, 
\end{equation}
\begin{equation}\label{inclusion:V_V_relaxed}
V\subset\mathcal{V}_0\subset\mathcal{V}. 
\end{equation}

Let $s$ be an initial time, $r\in (s,T]$, $y$ be an initial position of an agent agent, $\xi\in\mathcal{U}$, $\zeta\in\mathcal{V}$ be players' controls  acting upon an agent, $[s,r]\ni t\mapsto m(t)\in\ptd$ be a flow of probabilities. Denote by $x(\cdot,s,y,m(\cdot),\xi,\zeta)$ the unique solution of  the initial value problem
$$\frac{d}{dt}x(t)=\int_U\int_Vf(t,x(t),m(t),u,v)\xi(du|t)\zeta(dv|t),\ \ x(s)=y. $$ Let  $\operatorname{traj}_{m(\cdot)}^{s,r}: \td\times\mathcal{U}\times\mathcal{V}\rightarrow\mathcal{C}_{s,r} $ denote the operator assigning to $y$, $\xi$ and $\zeta$ the trajectory $x(\cdot,s,y,m(\cdot),\xi,\zeta)$.

Now we turn to the definition of evolution of the whole mean field type system driven by open-loop strategies. Recall that the state space of the whole system is $\ptd$. Since the agents placed in one point can be influenced by various controls, we use distribution of controls. They are probabilities on the product of phase space and space of controls. Using disintegration, one can regard them as weakly measurable functions of position taking values in the space of controls. Let $m\in\ptd$ be a distribution of agents. First, set
$$\mathcal{A}^c[m]\triangleq\mathcal{M}^+_{\text{product}}(\td,m,U),\ \ \mathcal{B}^c[m]\triangleq\mathcal{M}^+_{\text{product}}(\td,m,V). $$
Elements $\mathcal{A}^c[m]$ and $\mathcal{B}^c[m]$ are distributions of the constant controls of the first and second players respectively compatible with the measure $m$. 

In the same way, one can introduce the distribution of relaxed controls of the first and second players by the rules. Put
$$\mathcal{A}[m]\triangleq\mathcal{M}^+_{\text{product}}(\td,m,\mathcal{U}),\ \ \mathcal{B}[m]\triangleq\mathcal{M}^+_{\text{product}}(\td,m,\mathcal{V}). $$ Due to conventions (\ref{inclusion:U_U_relaxed}) and (\ref{inclusion:V_V_relaxed}), we assume that
$$\mathcal{A}^c[m]\subset \mathcal{A}[m],\ \ \mathcal{B}^c[m]\subset \mathcal{B}[m]. $$

The definition of feedback strategies proposed in~\cite{Averbokh_mfdg} requires the  notion of joint distributions of players' controls. A joint distribution of players' controls is a probability on $\td\times\mathcal{U}\times\mathcal{V}$ compatible with the distribution of agent $m$, i.e., the set of joint distribution of players controls is
$$\mathcal{D}[m]\triangleq \mathcal{M}^+_{\text{product}}(\td,m,\mathcal{U}\times\mathcal{V}). $$ If $\alpha\in \mathcal{A}[m]$, then denote by $\mathcal{D}_1[\alpha]$ the set of joint distribution of players' control produced by the reaction of the second players to $\alpha$, i.e.,
$$\mathcal{D}_1[\alpha]\triangleq \{\varkappa\in\mathcal{D}[m]:\operatorname{p}^{1,2}\sharp\varkappa=\alpha\}. $$ In the same way we define the set of joint distributions of players' controls compatible with the  distribution of first player's controls $\beta\in\mathcal{B}[m]$:
$$\mathcal{D}_2[\beta]\triangleq \{\varkappa\in\mathcal{D}[m]:\operatorname{p}^{1,3}\sharp\varkappa=\beta\}. $$

\begin{definition}\label{def:motion} Let $s,r\in[0,T]$, $s<r$, $m_*\in\ptd$ be an initial distribution of players, $\varkappa\in\mathcal{D}[m_*]$ be a joint distribution of players' controls consistent with $m_*$, we say that $[s,r]\ni t\mapsto m(t)\in\ptd$ is a flow of probabilities generated by $s$, $m_*$ and $\varkappa$ and write $m(\cdot)=m(\cdot,s,m_*,\varkappa)$ if there exists $\chi\in\mathcal{P}^2(\mathcal{C}_{s,r})$ such that
	\begin{itemize}
		\item $m(s)=m_*$;
		\item $m(t)=e_t\sharp\chi$;
		\item $\chi=\operatorname{traj}^{s,r}_{m(\cdot)}\sharp\varkappa$.
	\end{itemize}
\end{definition}

It can be proved in the same spirit as in~\cite{Sznitman} that there exists the unique flow of probabilities $m(\cdot,s,m_*,\varkappa)$.

It is natural to assume that the players can use information about the state of the whole system  to adjust their control. Recall that the state of the system in the mean field setting is provided by the distribution of agents. We use the Krasovskii-Subbotion approach~\cite{NN_PDG_en} which assumes that the controls is formed using the information about the state of the system in the finite number of time instants. 

\begin{definition}\label{def:first_player_strategy} Let $t_0$ be an initial time, $\Delta=\{t_i\}_{i=0}^N$ be a partition of $[t_0,T]$. A stepwise strategy with memory of the first player 
is a collection  $\mathbb{u}^\Delta=\{\mathbb{u}^\Delta_k\}_{k=0}^{N-1}$ where $\mathbb{u}^\Delta_k:(\ptd)^{k+1}\rightarrow\mathcal{P}^2(\td\times U)$ are such that 
$$\operatorname{p}^1\sharp\mathbb{u}^\Delta_k[m_0,\ldots,m_{k}]=m_{k}. $$\end{definition}
The value $\mathbb{u}^\Delta_k[m_0,\ldots,m_{k}]$  determines the distribution of controls used by the first player on the time interval $[t_{k},t_{k+1})$. 

\begin{definition}\label{def:stepwise_firts}
Let $t_0$ be an initial time,   $m_0$ be an initial position, $\Delta$ be a partition of $[t_0,T]$, $\mathbb{u}^\Delta=\{\mathbb{u}^\Delta_k\}_{k=0}^{N-1}$ be a stepwise strategy with memory of the first player. We say that a flow of probability $m(\cdot)$ is generated by $t_0$, $m_0$ and $\mathbb{u}^\Delta$ if there exists a sequence of probabilities $\alpha_k\in\mathcal{P}^2(\td\times U)$, $\varkappa_k\in \mathcal{P}^2(\td\times U\times\mathcal{V})$, $k=0,\ldots,N-1$, such that
\begin{itemize}
	\item $m(t_0)=m_0$;
	\item $\alpha_k\triangleq \mathbb{u}^\Delta_k[m(t_0),m(t_1),\ldots, m(t_k)]$, $\varkappa_k\in \mathcal{D}_1[\alpha_k]$;
	\item $m(t)=m(t,t_{k},m(t_k),\varkappa_k)$, $t\in [t_{k},t_{k+1}]$, $k=0,\ldots,N-1$.
\end{itemize}

We denote the set of flows of probabilities generated by  $t_0$, $m_0$, $\Delta$, and $\mathbb{u}^\Delta$ by 
$\mathcal{X}_1(t_0,m_0,\Delta,\mathbb{u}^\Delta)$.
\end{definition}

The second player's strategies are defined in the similar way. 
\begin{definition}
Given $t_0\in [0,T]$ and a partition of $[t_0,T]$ $\Delta=\{t_k\}_{k=0}^N$, a collection $\mathbb{u}^\Delta=\{\mathbb{u}^\Delta_k\}_{k=0}^{N-1}$ is called a stepwise strategy of the second player provided that $\mathbb{v}^\Delta_k:(\ptd)^{k+1}\rightarrow\mathcal{P}^2(\td\times V)$ are such that 
$$\operatorname{p}^1\sharp\mathbb{v}^\Delta_k[m_0,\ldots,m_{k}]=m_{k}. $$ 
\end{definition}
\begin{definition}\label{def:stepwise_second}
	We say that a flow of probabilities $m(\cdot)$ is generated by the initial time $t_0$, the initial distribution of agents $m_0$, the partition of $[t_0,T]$ $\Delta=\{t_k\}_{k=0}^N$ and the stepwise strategy with memory of the second player $\mathbb{v}^\Delta=\{\mathbb{v}^\Delta_k\}_{k=0}^{N-1}$ if one can find joint distributions of players' controls $\beta_k\in\mathcal{P}^2(\td\times V)$, $\varkappa_k\in\mathcal{P}^2(\td\times\mathcal{U}\times V)$, $k=0,\ldots, N-1$, such that
	\begin{itemize}
		\item $m(t_0)=m_0$;
		\item $\beta_k=\mathbb{v}^\Delta_k[m(t_0),\ldots,m(t_k)]$, $\varkappa_k\in\mathcal{D}_2[\beta_k]$;
		\item $m(t)=m(t,t_k,m(t_k),\varkappa_k)$, $t\in [t_k,t_{k+1}]$, $k=0,\ldots,N-1$.
	\end{itemize} 

Denote the set of flow of probabilities generated by $t_0$, $m_0$, $\Delta$ and $\mathbb{v}^\Delta$ by $\mathcal{X}_2(t_0,m_0,\Delta,\mathbb{v}^\Delta)$.
\end{definition}

Now let us briefly describe the feedback approach to the mean field type differential games~\cite{Averbokh_mfdg}. It implies that a feedback strategy of the first player is a mapping $\mathfrak{u}:[0,T]\times \ptd\rightarrow \mathcal{P}^2(\td\times U)$ such that
$$\operatorname{p}^1\sharp\mathfrak{u}[t,m]=m. $$ In this case, given a partition $\Delta=\{t_k\}_{k=0}^N$ of the time interval $[t_0,T]$, the  stepwise strategy $\mathbb{u}^\Delta=\{\mathbb{u}^\Delta_k\}_{k=0}^{N-1}$ realizing the feedback strategy $\mathfrak{u}$ is defined as
$$\mathbb{u}^\Delta_k[m_0,\ldots,m_k]\triangleq \mathfrak{u}[t_k,m_k]. $$ 
Analogously, a feedback strategy of the second player is a mapping $\mathfrak{v}:[0,T]\times \ptd\rightarrow \mathcal{P}^2(\td\times V)$ such that
$$\operatorname{p}^1\sharp\mathfrak{v}[t,m]=m. $$ Given a partition of the interval $[t_0,T]$ $\Delta=\{t_k\}_{k=0}^N$, one can reduce the feedback strategy to the strategies with memory letting
$$\mathbb{v}^\Delta_k[m_0,\ldots,m_k]\triangleq \mathfrak{v}[t_k,m_k]. $$ 

The upper and lower value functions by the rules
\begin{equation}\label{intro:upper_value}
\operatorname{Val}^+(t_0,m_0)=\inf_{\Delta,\mathbb{u}^\Delta}\sup_{m(\cdot)\in\mathcal{X}_1(t_0,m_0,\Delta,\mathbb{u}^\Delta)}g(m(T)),
\end{equation}
\begin{equation}\label{intro:lower_value}
\operatorname{Val}^-(t_0,m_0)=\sup_{\Delta,\mathbb{v}^\Delta}\inf_{m(\cdot)\in\mathcal{X}_2(t_0,m_0,\Delta,\mathbb{v}^\Delta)}g(m(T)).
\end{equation}

Allowing in (\ref{intro:upper_value}) only the stepwise realizations of feedback strategies, we obtain the upper value function in the class of feedback strategies. Denote it by $\operatorname{Val}^+_f(t_0,m_0)$. Analogously, reducing the class of strategies in (\ref{intro:lower_value}) to the  stepwise realizations of second player's feedback strategies, we arrive at the notion of lower value function in the class of feedback strategies $\operatorname{Val}^-_f(t_0,m_0)$.

Obviously,
$$\operatorname{Val}^-(t_0,m_0)\leq \operatorname{Val}^-_f(t_0,m_0)\leq \operatorname{Val}^+_f(t_0,m_0) \leq \operatorname{Val}^+(t_0,m_0). $$  It is proved (see \cite[Theorem 2]{Averbokh_mfdg})  that, under imposed conditions, the game has the value in the class of feedback strategies, i.e.,
$$\operatorname{Val}^-_f(t_0,m_0)=\operatorname{Val}^+_f(t_0,m_0)=\operatorname{Val}^+(t_0,m_0)=\operatorname{Val}^-(t_0,m_0)=\operatorname{Val}(t_0,m_0). $$ Hereinafter, $\operatorname{Val}$ denotes the value of the game. Below we estimate it using sub- and supersolutions of the Hamilton-Jacobi equations corresponding to the finite-dimensional differential games.

\section{Finite state mean field type Markov games}\label{sect:finite_state_games}
In this section, we consider the finite state  mean field type continuous-time  game. The main feature of the finite state mean field systems is that a phase variable is an element of a finite-dimensional simplex. Thus, the finite state mean field type Markov games can be reduced to the finite-dimensional differential game theory.

Let $\mathcal{S}$ be a finite set. We will examine the approximation of original mean field type differential game by the mean field type Markov game with the state space of each agent equal to $\mathcal{S}$. Thus, without loss of generality, we assume that $\mathcal{S}\subset \td$. Denote
by $d(\mathcal{S})$ the fineness of $\mathcal{S}$:
\begin{equation}\label{intro:fineness_S}
d(\mathcal{S})\triangleq \min_{\bar{x},\bar{y}\in\mathcal{S},\bar{x}\neq\bar{y}}\|\bar{x}-\bar{y}\|. 
\end{equation}
 Furthermore, let $\Sigma$ be a simplex on $\{1,\ldots,|\mathcal{S}|\}$: $$\Sigma\triangleq \left\{\mu=(\mu_{\bar{x}})_{\bar{x}\in\mathcal{S}}:\mu_{\bar{x}}\geq 0,\sum_{x\in \mathcal{S}}\mu_{\bar{x}}=1\right\}.$$ Notice that \[\Sigma\subset \mathbb{R}^{|\mathcal{S}|}.\]

Denote by $\mathbb{1}_{\bar{y}}$ the distribution concentrated at $\bar{y}$, i.e., $\mathbb{1}_{\bar{y}}=(\mathbb{1}_{\bar{y},\bar{x}})_{\bar{x}\in\mathcal{S}}$ where
	\begin{equation}\label{intro:1}
	\mathbb{1}_{\bar{y},\bar{x}}=\left\{\begin{array}{cc}
	1, & \bar{x}=\bar{y}, \\
	0, & \bar{x}\neq\bar{y}.
	\end{array}\right. 
	\end{equation}  The distributions $\mathbb{1}_{\bar{y}}$ are extremal points of $\Sigma$.
 
Given $p\geq 1$, the $p$-th metric  on $\mathbb{R}^{|\mathcal{S}|}$ (and, thus, on $\Sigma$)    is defined by the rule: for $\mu^1=(\mu_{\bar{x}}^1)_{\bar{x}\in\mathcal{S}},\mu^2=(\mu_{\bar{x}}^2)_{\bar{x}\in\mathcal{S}}\in \mathbb{R}^{|\mathcal{S}|}$,  $$\|\mu^1-\mu^2\|_p\triangleq \left[\sum_{x\in \mathcal{S}}|\mu_{\bar{x}}^1-\mu_{\bar{x}}^2|^p\right]^{1/p}. $$

If $\mu=(\mu_{\bar{x}})_{\bar{x}\in\mathcal{S}}\in\Sigma$, then let $\tilde{\mu}$ stand for the  measure on $\mathcal{S}$ defined by
\begin{equation}\label{intro:tilde_mu}
\tilde{\mu}\triangleq \sum_{\bar{x}\in \mathcal{S}}\mu_{\bar{x}}\delta_{\bar{x}}\in\mathcal{P}(\mathcal{S}).
\end{equation} The mapping 
$$\Sigma\ni\mu\mapsto \tilde{\mu}\in\mathcal{P}(\mathcal{S}) $$ is an isomorphism between $\Sigma$ and $\mathcal{P}(\mathcal{S})$. However, the metrics on $\Sigma$ and $\mathcal{P}(\mathcal{S})$ are not equivalent. The relation between them is given in the following.

\begin{proposition}\label{prop:metrics} There exist positive constants $C_1$ and $C_2$ (depending on $\mathcal{S}$) such that, if $\mu^1=(\mu_{\bar{x}}^1)_{\bar{x}\in\mathcal{S}},\ \ \mu^2=(\mu_{\bar{x}}^2)_{\bar{x}\in\mathcal{S}}\in\Sigma$, $\widetilde{\mu^1}=\sum_{\bar{x}\in \mathcal{S}}\mu_{\bar{x}}^1\delta_x$, $\widetilde{\mu^2}=\sum_{x\in \mathcal{S}}\mu_{\bar{x}}^2\delta_{\bar{x}}$, then
	\begin{equation}\label{ineq:R_d_wasser}
	\|\mu^1-\mu^2\|_p\leq C_1 W_p(\widetilde{\mu^1},\widetilde{\mu^2}), 
	\end{equation}
\begin{equation}\label{ineq:wasser_R_d}
	W_p(\widetilde{\mu^1},\widetilde{\mu^2})\leq C_2(\|\mu^1-\mu^2\|_p)^{1/p}. 
\end{equation}
\end{proposition}
This proposition is proved in the Appendix.

We regard the set $\mathcal{S}$ to be  the state space for each agent in the  mean field type Markov game, whereas $\Sigma$ is the set of distributions on $\mathcal{S}$. The dynamics of the finite state mean field type game is given by the controlled Markov chain with the transition rates determined by the Kolmogorov matrix
$$Q(t,\mu,u,v)=(Q_{\bar{x},\bar{y}}(t,\mu,u,v))_{\bar{x},\bar{y}\in\mathcal{S}},$$ where $ t\in [0,T],$ $\mu\in\Sigma,$ $u\in U,$ $v\in V.$  As above, the variable $u$ is  controlled by the first player, whereas $v$ is chosen by the second player. The mean field methodology implies that  $\mu(t)=\operatorname{Law}(X(t))$, where $X(t)$ is  the stochastic process produced by the nonlinear Markov chain with the Kolmogorov matrix $Q(t,\mu(t),u(t),v(t))$. We assumes that the first (second) player tries to minimize (maximize) the quantity \[g(\widetilde{\mu(T)}). \]

We impose the following conditions on the matrix $Q$: 
\begin{enumerate}[label=(M\arabic*)]
	\item\label{cond:M_Kolmogorov} for every $(t,\mu,u,v)\in [0,T]\times\Sigma\times U\times V$, $Q_{\bar{x},\bar{y}}(t,\mu,u,v)\geq 0$ when $\bar{x}\neq \bar{y}$ and
	$$\sum_{\bar{y}\in \mathcal{S}}Q_{\bar{x},\bar{y}}(t,\mu,u,v)=0; $$
	\item\label{cond:M_continuity} the functions $[0,T]\times \Sigma\times U\times V\ni (t,\mu,u,v)\mapsto Q_{\bar{x},\bar{y}}(t,\mu,u,v)\in \mathbb{R}$ are continuous;
	\item\label{cond:M_Lipschitz} there exists a constant $L'$ such that for any $t\in [0,T]$, $\bar{x},\bar{y}\in\mathcal{S}$, $\mu^1,\mu^2\in\Sigma$, $u\in U$, $v\in V$,  $$|Q_{\bar{x},\bar{y}}(t,\mu^1,u,v)-Q_{\bar{x},\bar{y}}(t,\mu^2,u,v)|\leq L'\|\mu^1-\mu^2\|_2 .$$ 
\end{enumerate}

As above, considering the finite state mean field type differential games, we first define the motion of the representative agent, then turn to the behavior of the whole system. On the first stage, we assume that the flow of probabilities $[s,r]\ni t\mapsto\mu(t)\in\Sigma$ is given, while the players use relaxed feedback strategies. For a Markov game, it is convenient to introduce relaxed feedback controls of the players using weakly measurable functions. 

Let 
\[\mathcal{U}_{\text{instant}}\triangleq \mathcal{P}(U),\ \ \mathcal{V}_{\text{instant}}\triangleq \mathcal{P}(V) \] be sets of players' controls applied at each time instant and each state from $\mathcal{S}$. Therefore, the instantaneous controls of the players lie in the sets \[\mathcal{U}_{\text{instant}}^{\mathcal{S}}=\{\gamma_{\mathcal{S}}=(\gamma_{\bar{x}})_{\bar{x}\in\mathcal{S}}:\gamma_{\bar{x}}\in\mathcal{U}_{\text{instant}}\}, \]
\[\mathcal{V}_{\text{instant}}^{\mathcal{S}}=\{\vartheta_{\mathcal{S}}=(\vartheta_{\bar{x}})_{\bar{x}\in\mathcal{S}}:\vartheta_{\bar{x}}\in\mathcal{V}_{\text{instant}}\}. \]

We say that a function $[s,r]\ni t\mapsto\gamma_{\mathcal{S}}(t)=(\gamma_{\bar{x}})_{\bar{x}\in\mathcal{S}}\in \mathcal{U}_{\text{instant}}^{\mathcal{S}}$ is a measurable control  of the first player in the finite state mean field type continuous-times game if each coordinate function $\gamma_{\bar{x}}(\cdot)$ is weakly measurable. The measurable controls of the second player are introduced in the same way.
Given measurable controls $\gamma_{\bar{x}}:[0,T]\rightarrow \mathcal{U}_{\text{instant}}$, $\vartheta_{\bar{x}}:[0,T]\rightarrow \mathcal{V}_{\text{instant}}$ of the first and second players respectively, we regard $\gamma_{\bar{x}}(t)$, $\vartheta_{\bar{x}}(t)$ as the instantaneous controls of the players acting upon the agent who occupy the state $\bar{x}$ at time $t$. 

Considering the relaxed feedback strategy, we arrive at the following Kolmogorov matrix:
\[\mathcal{Q}(t,\mu,\gamma_{\mathcal{S}},\vartheta_{\mathcal{S}})=(\mathcal{Q}_{\bar{x},\bar{y}}(t,\mu,\gamma_{\bar{x}},\vartheta_{\bar{x}}))_{\bar{x},\bar{y}\in\mathcal{S}}\] with
$$\mathcal{Q}_{\bar{x},\bar{y}}(t,\mu,\gamma_{\bar{x}},\vartheta_{\bar{x}})\triangleq \int_U\int_VQ_{\bar{x},\bar{y}}(t,\mu,u,v)\gamma_{\bar{x}}(du)\vartheta_{\bar{x}}(dv). $$

The probabilities of the states $\nu(t)=(\nu_{\bar{y}}(t))_{\bar{y}\in\mathcal{S}}$, $\nu_{\bar{y}}=P(X(t)=\bar{y})$ obeys the backward Kolmogorov equation
\[\frac{d}{dt}\nu_{\bar{y}}(t)=\sum_{\bar{x}\in\mathcal{S}}\nu_{\bar{x}}(t)\mathcal{Q}_{\bar{x},\bar{y}}(t,\mu(t),\gamma_{\bar{x}}(t),\vartheta_{\bar{x}}(t)),\ \ \bar{y}\in\mathcal{S}.\] Assuming that $\nu=(\nu_{\bar{x}})_{\bar{x}\in\mathcal{S}}$ is a row-vector with some ordering of~$\mathcal{S}$, we rewrite this equations in the vector form
\begin{equation}\label{eq:nu}
\frac{d}{dt}\nu(t)=\nu(t)\mathcal{Q}(t,\mu(t),\gamma_{\mathcal{S}}(t),\vartheta_{\mathcal{S}}(t)),\ \ \nu(s)=\nu_*. 
\end{equation}

The mean field game methodology implies that the probability of the event that an  agent occupies the state $\bar{x}$ at time $t$ is equal to $\mu_{\bar{x}}(t)$. This leads to the following equations:
$$\frac{d}{dt}\mu_{\bar{y}}(t)=\sum_{\bar{x}\in \mathcal{S}}\mu_{\bar{x}}(t)\mathcal{Q}_{\bar{x},\bar{y}}(t,\mu(t),\gamma_{\bar{x}}(t),\vartheta_{\bar{x}}(t)),\ \ \bar{y}\in\mathcal{S} $$ or in the vector form 
\begin{equation}\label{eq:backward_Kolmogorov}
\frac{d}{dt}\mu(t)=\mu(t)\mathcal{Q}(t,\mu(t),\gamma_{\mathcal{S}}(t),\vartheta_{\mathcal{S}}(t)).
\end{equation} 

Equation (\ref{eq:backward_Kolmogorov}) describes the dynamics of finite-dimensional differential game with the controls of the players $\gamma_{\mathcal{S}}(t)$, $\vartheta_{\mathcal{S}}(t)$. Recall that the first players tries to minimize 
\begin{equation}\label{expression:g_mu}
\hat{g}(\mu(T)), 
\end{equation} where
\begin{equation}\label{intro:hat_g}
\hat{g}(\mu)\triangleq g(\widetilde{\mu})=g\left(\sum_{\bar{x}\in \mathcal{S}}\mu_{\bar{x}}\delta_{\bar{x}}\right).
\end{equation}
 The purpose of the second player is assumed to be opposite. 

The solution of this differential game can be described using the notion of the viscosity/minimax solution of the corresponding Hamilton-Jacobi equation. Let $t\in [0,T]$, $\mu=(\mu_{\bar{x}})_{\bar{x}\in\mathcal{S}}\in\Sigma$, $w=(w_{\bar{x}})_{\bar{x}\in\mathcal{S}}\in\mathbb{R}^{\mathcal{S}}$. Put
\begin{equation*}
\begin{split}
\mathcal{H}^{\mathcal{Q}}(t,\mu,w)&\triangleq \min_{\gamma_{\mathcal{S}}\in\mathcal{U}_{\text{instant}}^{\mathcal{S}}}\max_{\vartheta_{\mathcal{S}}\in\mathcal{V}_{\text{instant}}^{\mathcal{S}}}\sum_{\bar{y}\in \mathcal{S}}\sum_{\bar{x}\in \mathcal{S}}\mu_{\bar{x}}\mathcal{Q}_{\bar{x},\bar{y}}(t,\mu,\gamma_{\bar{x}},\vartheta_{\bar{x}})w_{\bar{y}} \\&=\min_{\gamma_{\mathcal{S}}\in\mathcal{U}_{\text{instant}}^{\mathcal{S}}}\max_{\vartheta_{\mathcal{S}}\in\mathcal{V}_{\text{instant}}^{\mathcal{S}}}\sum_{\bar{x}\in \mathcal{S}}\mu_{\bar{x}}\sum_{\bar{y}\in \mathcal{S}}\mathcal{Q}_{\bar{x},\bar{y}}(t,\mu,\gamma_{\bar{x}},\vartheta_{\bar{x}})w_{\bar{y}}.
\end{split}
\end{equation*} Since we admit the feedback controls $\gamma_{\mathcal{S}}=(\gamma_{\bar{x}})_{\bar{x}\in\mathcal{S}}$, $\vartheta_{\mathcal{S}}=(\vartheta_{\bar{x}})_{\bar{x}\in\mathcal{S}}$, we have that
\begin{equation*}
\mathcal{H}^{\mathcal{Q}}(t,\mu,w)=\sum_{\bar{x}\in \mathcal{S}}\mu_{\bar{x}} \min_{\gamma_{\bar{x}}\in\mathcal{U}_{\text{instant}}}\max_{\vartheta_{\bar{x}}\in\mathcal{V}_{\text{instant}}}\sum_{\bar{y}\in \mathcal{S}}\mathcal{Q}_{\bar{x},\bar{y}}(t,\mu,\gamma_{\bar{x}},\vartheta_{\bar{x}})w_{\bar{y}}.
\end{equation*}
Furthermore, by construction of $\mathcal{Q}$, we have
\[\sum_{\bar{y}\in \mathcal{S}}\mathcal{Q}_{\bar{x},\bar{y}}(t,\mu,\gamma_{\bar{x}},\vartheta_{\bar{x}})w_{\bar{y}}=\int_{U}\int_V\sum_{\bar{y}\in \mathcal{S}}Q_{\bar{x},\bar{y}}(t,\mu,u,v)w_{\bar{y}}\gamma_{\bar{x}}(du)\vartheta_{\bar{x}}(dv).\]
Thus, the function \[\sum_{\bar{y}\in \mathcal{S}}\mathcal{Q}_{\bar{x},\bar{y}}(t,\mu,\gamma_{\bar{x}},\vartheta_{\bar{x}})w_{\bar{y}}\] is convex w.r.t. $\gamma_{\bar{x}}$ and concave w.r.t. $\vartheta_{\bar{x}}$. From the minimax theorem \cite{Sion_minimax}, it follows that
\[
\begin{split}
\mathcal{H}^{\mathcal{Q}}(t,\mu,w)&=\sum_{\bar{x}\in \mathcal{S}}\mu_{\bar{x}} \max_{\vartheta_{\bar{x}}\in\mathcal{V}_{\text{instant}}}\min_{\gamma_{\bar{x}}\in\mathcal{U}_{\text{instant}}}\sum_{\bar{y}\in \mathcal{S}}\mathcal{Q}_{\bar{x},\bar{y}}(t,\mu,\gamma_{\bar{x}},\vartheta_{\bar{x}})w_{\bar{y}}\\
&= \max_{\vartheta_{\mathcal{S}}\in\mathcal{V}_{\text{instant}}^{\mathcal{S}}}\min_{\gamma_{\mathcal{S}}\in\mathcal{U}_{\text{instant}}^{\mathcal{S}}}\sum_{\bar{y}\in \mathcal{S}}\sum_{\bar{x}\in \mathcal{S}}\mu_{\bar{x}}\mathcal{Q}_{\bar{x},\bar{y}}(t,\mu,\gamma_{\bar{x}},\vartheta_{\bar{x}})w_{\bar{y}}
\end{split}
\]

In particular, we showed that differential game (\ref{eq:backward_Kolmogorov}), (\ref{expression:g_mu}) satisfies the Isaacs' condition.

For $\varphi:[0,T]\times\Sigma\rightarrow \mathbb{R}$, we denote by $\nabla\varphi(t,\mu)$ the vector of partial derivatives $$\nabla\varphi(t,\mu)\triangleq \left(\frac{\partial\varphi}{\partial\mu_{\bar{x}}}(t,\mu)\right)_{\bar{x}\in\mathcal{S}}.$$ 

Consider the Hamilton-Jacobi equations
\begin{equation*}\label{eq:HJ}
\frac{\partial\varphi}{\partial t}+\mathcal{H}^{\mathcal{Q}}(t,\mu,\nabla\varphi))=0.\ \ \varphi(T,\mu)=\hat{g}(\mu),
\end{equation*} where $\hat{g}$ is introduced by (\ref{intro:hat_g}).

Notice that a supersolutions (respectively, subsolution) of equation~(\ref{eq:HJ}) is an upper (lower) bound for the  value function for the differential game~(\ref{eq:backward_Kolmogorov}),~(\ref{expression:g_mu}) \cite{Subb_book}. The solution of~(\ref{eq:HJ}) is the value function of differential game~(\ref{eq:backward_Kolmogorov}),~(\ref{expression:g_mu}) (see \cite{Subb_book}).

\section{Main result}\label{sect:result}
The purpose of this Section is to formulate Theorem~\ref{th:existence_strategy} which states that, if the set $\mathcal{S}$ and the Markov chain with the Kolmogorov matrix $Q(t,\mu,u,v)$ approximate the phase space for the original mean field type differential game $\td$ and the dynamics $f$, then  supersolutions (respectively, subsolutions) of (\ref{eq:HJ})  provide upper (respectively, lower) estimates of the players' outcomes in the original mean field type differential game. This also yields the estimates of the value function of the mean field type differential game by the value function of finite-dimensional differential game (see Corollary~\ref{corollary:value}).

First, let $\ell:\td\times\td\rightarrow\rd$ be a measurable function assigning to a pair of elements $x,y\in\td$ a vector $z'\in x-y$ of the minimal norm. The existence of this function follows from the measurable maximum theorem \cite[Theorem 18.19]{Infinite_dimensional_analysis}. Notice that, for $x,y,z\in\td$,
\begin{equation}\label{equality:ell}
\|x-y\|^2-\|y-z\|^2=\|x-z\|^2-2\langle \ell(x,z),\ell(y,z)\rangle.
\end{equation}

Recall that we assume that $\mathcal{S}\subset \td$. Let $\varepsilon>0$ be such that 
  \begin{equation}\label{assumption:ineq:S_td}
\max_{x\in\td}\min_{\bar{y}\in\mathcal{S}}\|x-\bar{y}\|\leq \varepsilon, 
\end{equation}
  
\begin{equation}\label{assumption:ineq:drift}
\max_{t\in [0,T],\bar{x}\in\mathcal{S},\mu\in\Sigma,u\in U,v\in V}\left\|f(t,\bar{x},\tilde{\mu},u,v)-\sum_{\bar{y}\in\mathcal{S},\bar{y}\neq \bar{x}}\ell(\bar{y},\bar{x})Q_{\bar{x},\bar{y}}(t,\mu,u,v)\right\|\leq \varepsilon,
\end{equation} and
\begin{equation}\label{assumption:ineq:violance}
\max_{t\in [0,T],\bar{x}\in\mathcal{S},\mu\in\Sigma,u\in U,v\in V}\sum_{\bar{y}\in\mathcal{S}}\|\bar{y}-\bar{x}\|^2Q_{\bar{x},\bar{y}}(t,\mu,u,v)\leq \varepsilon^2.
\end{equation}

Below, without loss of generality, we assume that $\varepsilon\leq 1$.

Further, given $m\in\ptd$, denote by $\operatorname{pr}_{\mathcal{S}}(m)$ an element of $\Sigma$ such that $\widetilde{\operatorname{pr}_{\mathcal{S}}(m)}$ is a proximal to $m$ element of $\mathcal{P}^2(\mathcal{S})$. This means that $ \operatorname{pr}_{\mathcal{S}}(m)$ minimizes the function \[\Sigma\ni \mu=(\mu_{\bar{x}})_{\bar{x}\in\mathcal{S}}\mapsto W_2(\tilde{\mu},m)=W_2\left(\sum_{\bar{x}\in\mathcal{S}}\mu_{\bar{x}}\delta_{\bar{x}},m\right).\] As above, due to the measurable maximum theorem (see \cite[Theorem 18.19]{Infinite_dimensional_analysis}) one may choose the mapping $m\mapsto \operatorname{pr}_{\mathcal{S}}(m)$ to be measurable.

Let $\varsigma_g$ be a modulus of continuity for the function $g$, i.e.,
\[\varsigma_g(\theta)\triangleq \sup\{|g(m')-g(m'')|:\, m',m''\in\ptd,\, W_2(m',m'')\leq \theta\}.\] Set 
\[C^*\triangleq \sqrt{1+2T}e^{(2L+1/2)T}.\] Here $L$ is a Lipschitz constant of the function $f$ w.r.t. $x$ and $m$ (see (\ref{intro:L})).

\begin{theorem}\label{th:existence_strategy} 
	
	\begin{enumerate}[label=(\roman{*})]
		
		\item Given a supersolution of (\ref{eq:HJ}) $\varphi^+ $, $t_0\in [0,T]$,  and a partition of $[t_0,T]$ $\Delta$, there exists a first player's strategy with memory  $\mathbb{u}^\Delta$ such that, for every $m_0\in\ptd$ and $m(\cdot)\in\mathcal{X}_1(t_0,m_0,\Delta,\mathbb{u}^\Delta)$, $$g(m(T))\leq \varphi^+(t_0,\operatorname{pr}_{\mathcal{S}}(m_0))+\varsigma_g(C^*\varepsilon+\varsigma^*(d(\Delta))).$$ 
		\item For a function $\varphi^-$ that is a subsolution of (\ref{eq:HJ}), an initial time $t_0$ and a partition of $[t_0,T]$ $\Delta$, one can construct a strategy with memory of the second player $\mathbb{v}^\Delta$, satisfying the following: for every $m_0\in\ptd$ and  $m(\cdot)\in \mathcal{X}_2(t_0,m_0,\Delta,\mathbb{v}^\Delta)$,
		$$g(m(T))\geq \varphi^-(t_0,\operatorname{pr}_{\mathcal{S}}(m_0))-\varsigma_g(C^*\varepsilon+\varsigma^*(d(\Delta))). $$
	\end{enumerate}
Here  $\varsigma^*(\epsilon)$ is a function taking values in $[0,+\infty)$ satisfying
$\varsigma^*(\epsilon)\rightarrow 0$ as $\epsilon\rightarrow 0$.
\end{theorem}

This theorem is proved in Sections~\ref{sect:construction_suboptimal}--\ref{sect:extremal_shift}.

Furthermore, Theorem~\ref{th:existence_strategy} implies the following estimates on the value function of the mean field type differential game.
\begin{corollary}\label{corollary:value} Let $\varphi^+$ and $\varphi^-$ be a supersolution and a subsolution of (\ref{eq:HJ}) respectively. If $t_0\in [0,T]$, $m_0\in\ptd$, then
	$$\varphi^-(t_0,\operatorname{pr}_{\mathcal{S}}(m_0))-\varsigma_g(C^*\varepsilon)\leq \operatorname{Val}(t_0,m_0)\leq \varphi^+(t_0,\operatorname{pr}_{\mathcal{S}}(m_0))+\varsigma_g(C^*\varepsilon). $$ 
	If $\varphi$ is the viscosity/minimax solution of (\ref{eq:HJ}), then
	\[
	|\operatorname{Val}(t_0,m_0)-\varphi(t_0,\operatorname{pr}_{\mathcal{S}}(m_0))|\leq \varsigma_g(C^*\varepsilon).
	\]
\end{corollary}

\section{Design of the approximating Markov chain}\label{sect:example_construction}
In this section, given a system (\ref{sys:agent_mfdg}), we construct an approximating Markov chain.
Here we assume additionally that the function $f$ is Lipschitz continuous with respect to $m$ in $1$-Wasserstein metric, i.e., for every $t\in [0,T]$, $x\in\td$, $m_1,m_2\in\mathcal{P}^1(\td)=\ptd$, $u\in U$, $v\in V$,
$$\|f(t,x,m_1,u,v)-f(t,x,m_2,u,v)\|\leq L'' W_1(m_1,m_2) $$ for some constant $L''$.

Let $h>0$ be such that $1/h\in\mathbb{N}$. Put $$\mathcal{S}_h\triangleq h\mathbb{Z}^d\cap \td.$$ To define the approximating Markov chain, we  write the dynamics in the coordinatewise form $$f(t,x,m,u,v)=(f_1(t,x,m,u,v),\ldots,f_d(t,x,m,u,v)).$$ Let $e^i$ stand for the  $i$-th coordinate vector.
The approximating Markov chain is defined as follows. For $t\in [0,T]$, $\bar{x},\bar{y}\in\mathcal{S}_h$, $\mu\in\Sigma$, $u\in U$, $v\in V$, set
\begin{equation}\label{intro:Kolmogorov_f_general}
Q^h_{\bar{x},\bar{y}}(t,\mu,u,v)\triangleq \left\{\begin{array}{cc}
\frac{1}{h}|f_i(t,x,\tilde{\mu},u,v)|, & \bar{y}=\bar{x}+h\operatorname{sgn}(f_i(t,x,\tilde{\mu},u,v))e^i, \\
-\frac{1}{h}\sum_{j=1}^d|f_j(t,x,\tilde{\mu},u,v)|, & \bar{x}=\bar{y}, \\
0, & \text{otherwise}.
\end{array}\right. 
\end{equation}

\begin{proposition}\label{prop:Kolmogorov_lattice} The matrix $Q$ defined by (\ref{intro:Kolmogorov_f_general}) satisfies condition~\ref{cond:M_Kolmogorov}--\ref{cond:M_Lipschitz} and (\ref{assumption:ineq:S_td})--(\ref{assumption:ineq:violance}) with \begin{equation}\label{intro:equality:vareps:approx:lattiice}
	\varepsilon=\sqrt{h}\cdot\max\left\{\sqrt{R}\sqrt[4]{d},\frac{\sqrt[4]{d}}{\sqrt{2}}\right\}.
	\end{equation}
\end{proposition}
\begin{proof}
First, let us check whether the matrix $Q^h$ satisfies conditions~\ref{cond:M_Kolmogorov}--\ref{cond:M_Lipschitz}. Indeed, direct calculations yields that
$$\sum_{\bar{y}\in\mathcal{S}_h,\bar{y}\neq\bar{x}}Q^h_{\bar{x},\bar{y}}(t,\mu,u,v)=\sum_{i=1}^d\frac{1}{h}|f_i(t,x,\tilde{\mu},u,v)|-\frac{1}{h}\sum_{j=1}^d|f_j(t,x,\tilde{\mu},u,v)|=0. $$ This gives condition~\ref{cond:M_Kolmogorov}. Condition~\ref{cond:M_continuity} follows from very definition of the matrix $Q^h$. Now, recall that we additionally assumed that the function $f$ is Lipschitz continuous with respect to the 1-Wasserstein metric. This implies that
$$|Q^h_{\bar{x},\bar{y}}(t,\mu^1,u,v)-Q^h_{\bar{x},\bar{y}}(t,\mu^2,u,v)|\leq \frac{L''}{h}W_1(\widetilde{\mu^1},\widetilde{\mu^2}). $$ From Proposition~\ref{prop:metrics} it follows that
$$|Q^h_{\bar{x},\bar{y}}(t,\mu^1,u,v)-Q^h_{\bar{x},\bar{y}}(t,\mu^2,u,v)|\leq \frac{C_2 L''}{h}\|\mu^1-\mu^2\|_1.$$ Using Holder's inequality we deduce condition~\ref{cond:M_Lipschitz}.

Now let us show that the constructed Markov chain approximated the original mean field type control system. We have that
$$\max_{x\in\td}\min_{\bar{y}\in\mathcal{S}}\|x-\bar{y}\|\leq \sqrt{d}h/2. $$ Further, let $t\in [0,T]$ $\bar{x}\in\mathcal{S}$, $\mu\in\Sigma$, $u\in U$, $v\in V$. We have that
\begin{equation*}\begin{split}
\sum_{\bar{y}\in\mathcal{S}_h,\bar{y}\neq\bar{x}}(\bar{y}-\bar{x})Q^h_{\bar{x},\bar{y}}(t,\mu,u,v)&= \sum_{i=1}^dh\operatorname{sgn}(f_i(t,x,\tilde{\mu},u,v))\frac{1}{h}|f_i(t,x,\tilde{\mu},u,v)|e^i\\ &=
\sum_{i=1}^{d}f_i(t,x,\tilde{\mu},u,v)e^i=f(t,x,\tilde{\mu},u,v).
\end{split}
\end{equation*} Finally, we have that
\begin{equation*}\begin{split}
\sum_{\bar{y}\in\mathcal{S}_h}\|\bar{y}-\bar{x}\|^2Q^h_{\bar{x},\bar{y}}(t,\mu,u,v)&= \sum_{i=1}^d\frac{h^2}{h}|f_i(t,x,\tilde{\mu},u,v)|\\ &\leq
h\sqrt{d}\left[\sum_{i=1}^{d}|f_i(t,x,\tilde{\mu},u,v)|^2\right]^{1/2}\leq R\sqrt{d}h,
\end{split}
\end{equation*} where the constant $R$ estimates the norm of $f$ (see (\ref{intro:R})).
 Thus, conditions (\ref{assumption:ineq:S_td})--(\ref{assumption:ineq:violance})  hold with  $\varepsilon$ determined by (\ref{intro:equality:vareps:approx:lattiice}).
\end{proof}

For the Markov chain with Kolmogorov matrix given by (\ref{intro:Kolmogorov_f_general}), we have that
\[\begin{split}
\mathcal{H}^{\mathcal{Q}}(t,\mu,w)=\frac{1}{h}\sum_{\bar{x}\in \mathcal{S}_h}\min_{\gamma_{\bar{x}}\in\mathcal{U}_{\text{instant}}}\max_{\vartheta_{\bar{x}}\in\mathcal{V}_{\text{instant}}}\int_U\int_V
	\sum_{i=1}^d |f_i&(t,\bar{x},\mu,u,v)|\\(w_{\bar{x}+he^i\operatorname{sgn}(f_i(t,\bar{x},\mu,u,v))}&-w_{\bar{x}})\gamma_{\bar{x}}(du)\vartheta_{\bar{x}}(dv).\end{split}
	\]

The computations of the Hamiltonian are simplified when we assume that 
$$f(t,x,m,u,v)=f^1(t,x,m,u)+f^2(t,x,m,v). $$ Here, we also assume that
$f^1$ and $f^2$ are Lipschitz continuous with respect to $m$ in $W_1$.
 As above we put $\mathcal{S}_h\triangleq \td\cap h\mathbb{Z}^d$. Representing the functions $f^1$ and $f^2$ in the coordinate-wise form
 $$f^1(t,x,m,u)=(f^1_1(t,x,m,u),\ldots, f^1_d(t,x,m,u)), $$
  $$f^2(t,x,m,v)=(f^2_1(t,x,m,v),\ldots, f^2_d(t,x,m,v)), $$ 
we introduce two Kolmogorov matrices for the dynamics $f^1$ and $f^2$. Set
\begin{equation*}\label{intro:Kolmogorov_f_1}
Q^{1,h}h_{\bar{x},\bar{y}}(t,\mu,u)\triangleq \left\{\begin{array}{cc}
\frac{1}{h}|f^1_i(t,x,\tilde{\mu},u)|, & \bar{y}=\bar{x}+h\operatorname{sgn}(f_i^1(t,x,\tilde{\mu},u))e^i, \\
-\frac{1}{h}\sum_{j=1}^d|f_j^1(t,x,\tilde{\mu},u)|, & \bar{x}=\bar{y}, \\
0, & \text{otherwise},
\end{array}\right. 
\end{equation*}
\begin{equation*}\label{intro:Kolmogorov_f_2}
Q^{2,h}h_{\bar{x},\bar{y}}(t,\mu,v)\triangleq \left\{\begin{array}{cc}
\frac{1}{h}|f^2_i(t,x,\tilde{\mu},v)|, & \bar{y}=\bar{x}+h\operatorname{sgn}(f_i^2(t,x,\tilde{\mu},v))e^i, \\
-\frac{1}{h}\sum_{j=1}^d|f_j^2(t,x,\tilde{\mu},v)|, & \bar{x}=\bar{y}, \\
0, & \text{otherwise},
\end{array}\right. 
\end{equation*}

\begin{equation}\label{intro:Kolmogorov_sum}
Q^{h}_{\bar{x},\bar{y}}(t,\mu,u,v)\triangleq Q^{1,h}_{\bar{x},\bar{y}}(t,\mu,u)+Q^{2,h}_{\bar{x},\bar{y}}(t,\mu,v).
\end{equation}

\begin{proposition}\label{prop:Kolmogorov_sum} The Kolmogorov matrix defined by (\ref{intro:Kolmogorov_sum}) satisfies conditions~\ref{cond:M_Kolmogorov}--\ref{cond:M_Lipschitz} and (\ref{assumption:ineq:S_td})--(\ref{assumption:ineq:violance}) with $\varepsilon$ given by (\ref{intro:equality:vareps:approx:lattiice}).
\end{proposition}
The proof of this proposition is similar to the proof of Proposition~\ref{prop:Kolmogorov_lattice}.
In this case, we have that
\[\begin{split}
\mathcal{H}^{\mathcal{Q}}(t,\mu,w)=&\frac{1}{h}\sum_{\bar{x}\in \mathcal{S}_h}\min_{\gamma_{\bar{x}}\in\mathcal{U}_{\text{instant}}}\int_U
\sum_{i=1}^d |f_i^1(t,\bar{x},\mu,u)|(w_{\bar{x}+he^i\operatorname{sgn}(f_i^1(t,\bar{x},\mu,u))}-w_{\bar{x}})\gamma_{\bar{x}}(du)\\&+\frac{1}{h}\sum_{\bar{x}\in \mathcal{S}_h}\max_{\vartheta_{\bar{x}}\in\mathcal{V}_{\text{instant}}}\int_V
\sum_{i=1}^d |f_i^2(t,\bar{x},\mu,v)|(w_{\bar{x}+he^i\operatorname{sgn}(f_i^2(t,\bar{x},\mu,v))}-w_{\bar{x}})\vartheta_{\bar{x}}(dv)
.\end{split}
\]

\section{Properties of mean field type controlled processes}\label{sect:properties}
In this section, we study the properties of the original and model  systems. The results proved below are used in the proof of Theorem~\ref{th:existence_strategy}.

First, let us mention the following property of the flow of probabilities generated by the original mean field type dynamics.

\begin{lemma}\label{lm:deterministic_motion_estimate}
	If  $m(\cdot)=m(\cdot,s,m_*,\varkappa)$ for some $s$, $m_*\in\ptd$, and distribution of controls $\varkappa\in\mathcal{D}[m_*]$, then 
	\[W_2^2(m(t),m(s))\leq R^2(t-s)^2.
	\]
\end{lemma}
\begin{proof}
	The conclusion of the lemma follows directly from  definition (\ref{def:motion}) and the estimate $\|f(t,x,u,v)\|\leq R$.
\end{proof}

We obtain the properties of the approximating system using the nonlinear Markov processes. To define this object,  it is convenient to introduce the generators (see~\cite{fleming_soner},~\cite{Kol_book},~\cite{Kol_markov}).  The Markov chain with Kolmogorov matrix $\mathcal{Q}(t,\mu,\gamma_{\mathcal{S}},\vartheta_{\mathcal{S}})$ corresponds to the generator $\Lambda_t^{\mathcal{Q}}[\mu,\gamma_{\mathcal{S}},\vartheta_{\mathcal{S}}]$ acting on $C(\mathcal{S})$ by the rule: if $\phi\in C(\mathcal{S})$, $\bar{x}\in\mathcal{S}$, then
\begin{equation}\label{intro:generator_Markov}
\Lambda_t^{\mathcal{Q}}[\mu,\gamma_{\mathcal{S}},\vartheta_{\mathcal{S}}]\phi(\bar{x})\triangleq \sum_{\bar{y}\in \mathcal{S}}[\phi(\bar{y})-\phi(\bar{x})]\mathcal{Q}_{\bar{x},\bar{y}}(t,\mu,\gamma_{\bar{x}},\vartheta_{\bar{x}}).
\end{equation}

\begin{definition}\label{def:Markov_motion}
	Given an initial distribution $\nu_*=(\nu_{*,\bar{x}})_{\bar{x}\in\mathcal{S}}$, a function $[s,r]\ni t\mapsto\mu(t)\in\Sigma$, players' strategies $[s,r]\ni t\mapsto \gamma_{\mathcal{S}}(t)\in\mathcal{U}_{\text{instant}}^{\mathcal{S}}$ and $[s,r]\ni t\mapsto \vartheta_{\mathcal{S}}(t)\in\mathcal{V}_{\text{instant}}^{\mathcal{S}}$, we say that $(\Omega,\mathcal{F},\{\mathcal{F}_t\}_{t\in [s,r]},P,X)$ is a Markov chain generated by  the Kolmogorov matrix $\mathcal{Q}(t,\mu(t),\gamma_{\mathcal{S}}(t),\vartheta_{\mathcal{S}}(t))$ if 
	\begin{itemize}
		\item $(\Omega,\mathcal{F},\{\mathcal{F}_t\}_{t\in [s,r]},P)$ is a filtered probability space;
		\item $X(\cdot)$ is a $\{\mathcal{F}_t\}_{t\in [s,r]}$-adapted stochastic  taking vales in $\mathcal{S}$;
		\item for every function $\phi\in C(\mathcal{S})$, 
		\begin{equation}\label{property:phi_martinglate}
		\phi(X(t))-\int_{s}^{t}\Lambda^{\mathcal{Q}}_\tau[\mu(\tau),\gamma_{\mathcal{S}}(\tau),\vartheta_{\mathcal{S}}(\tau)]\phi(X(\tau))d\tau 
		\end{equation} is a $\{\mathcal{F}_t\}_{t\in [s,r]}$-martingale; 
		\item $P(X(s)=\bar{x})=\nu_{*,\bar{x}}$.
	\end{itemize} Hereinafter, we denote by $\mathbb{E}$  the expectation according to the probability $P$.
\end{definition}

\begin{remark}\label{remark:existence_skorokhod} From \cite{Kol_markov} it follows that there exists at least one Markov chain generated by the Kolmogorov matrix  $\mathcal{Q}(t,\mu(t),\gamma_{\mathcal{S}}(t),\vartheta_{\mathcal{S}}(t))$.
	Furthermore, one can  assume that $\Omega=D([s,r],\mathcal{S})$, $\mathcal{F}$ is equal to the family of Borel sets on $D([s,r],\mathcal{S})$, whilst $\mathcal{F}_t$ is the family of sets such that their natural projections on $D([s,t],\mathcal{S})$ are Borel, whereas their natural projections on $[t,s]$ are equal to $D([t,r],\mathcal{S})$. Here $D([s,r],\mathcal{S})$ stands the Skorokhod space of c\`{a}dl\`{a}g functions on $[s,r]$ taking values in $\mathcal{S}$.\end{remark}

Notice that, if we denote \[\nu_{\bar{x}}(t)\triangleq P(X(t)=\bar{x}),\] then $\nu(\cdot)$ satisfies the forward Kolmogorov differential equation (\ref{eq:nu}). In particular, if $\nu_*=\mu(s)$, then $\nu(\cdot)=\mu(\cdot)$.

Set
\begin{equation*}\label{intro:R_1}
R_1\triangleq 1+2(1+Rd),
\end{equation*}
\begin{equation*}\label{intro:varsigma_1}
\varsigma_1(\epsilon)\triangleq\frac{4}{3}(1+R)R_1\epsilon^{1/2}.
\end{equation*} Notice that $\varsigma_1(\epsilon)\rightarrow 0$ as $\epsilon\rightarrow 0$. 

\begin{lemma}\label{lm:estimate} Let $s,r\in [0,T]$, $s<r$, $\gamma_{\mathcal{S}}(\cdot)$ and $\vartheta_{\mathcal{S}}(\cdot)$ be measurable functions taking values in $\mathcal{U}_{\text{instant}}^{\mathcal{S}}$ and $\mathcal{V}_{\text{instant}}^{\mathcal{S}}$ respectively, and let $\bar{z}\in\mathcal{S}$. If $(\Omega,\mathcal{F},\{\mathcal{F}_t\}_{t\in [s,r]},P,X)$ is a Markov chain generated by  the Kolmogorov matrix $\mathcal{Q}(t,\mu(t),\gamma_{\mathcal{S}}(t),\vartheta_{\mathcal{S}}(t))$ and initial distribution at time $s$ equal to $\mathbb{1}_{\bar{z}}$, then 
	$$\mathbb{E}\|X(t)-\bar{z}\|^2\leq R_1^2(t-s), $$ and
	$$\mathbb{E}\|X(t)-\bar{z}\|^2\leq \varepsilon^2 (t-s)+\varsigma_1(t-s)\cdot (t-s). $$
\end{lemma}
\begin{proof}
	Consider the function $q_{\bar{z}}(x)\triangleq \|x-\bar{z}\|^2$. Due to Definition \ref{def:Markov_motion}, we have that 
	\begin{equation}\label{equality:expect_q}
	\mathbb{E}q_{\bar{z}}(X(t))=\int_s^t\mathbb{E}\Lambda^{\mathcal{Q}}_\tau q_{\bar{z}}(X(\tau))d\tau
	\end{equation} 
	Further, using definition of the generator $\Lambda^{\mathcal{Q}}$ (see (\ref{intro:generator_Markov})) and equality (\ref{equality:ell}), we obtain the equality:
	\begin{equation}\label{ineq:Lambda_Q}
	\begin{split}
	\mathbb{E}\Lambda^{\mathcal{Q}}_\tau &q_{\bar{z}}(X(\tau))
	\\&=
	\sum_{\bar{x}\in \mathcal{S}}\nu_{*,\bar{x}}(\tau)\sum_{\bar{y}\in \mathcal{S},\bar{y}\neq\bar{x}}(\|\bar{y}-\bar{z}\|^2-\|\bar{x}-\bar{z}\|^2)\mathcal{Q}_{\bar{x},\bar{y}}(\tau,\mu(\tau),\gamma_{\bar{x}}(\tau),\vartheta_{\bar{x}}(\tau))\\&=
	\sum_{\bar{x}}\nu_{\bar{x}}(\tau)\sum_{\bar{y}\in \mathcal{S}}(\|\bar{y}-\bar{x}\|^2-2\langle \ell(\bar{y},\bar{x}),\ell(\bar{z},\bar{x})\rangle)\mathcal{Q}_{\bar{x},\bar{y}}(\tau,\mu(\tau),\gamma_{\bar{x}}(\tau),\vartheta_{\bar{x}}(\tau)).
	\end{split}
	\end{equation}
	
	Using  (\ref{assumption:ineq:violance}), we conclude  that
	\begin{equation}\label{ineq:violence_y_squared}
	\sum_{\bar{y}\in \mathcal{S}}\|\bar{y}-\bar{x}\|^2\mathcal{Q}_{\bar{x},\bar{y}}(\tau,\mu(\tau),\gamma_{\bar{x}}(\tau),\vartheta_{\bar{x}}) \leq  \varepsilon^2.
	\end{equation}
	To evaluate the last term, 
	recall that
	\begin{equation}\label{equality:Q}
	\mathcal{Q}_{\bar{x},\bar{y}}(\tau,\mu(\tau),\gamma_{\bar{x}}(\tau),\vartheta_{\bar{x}}(\tau))=\int_U\int_V Q_{\bar{x},\bar{y}}(\tau,\mu(\tau),u,v)\gamma_{\bar{x}}(\tau,du)\vartheta_{\bar{x}}(\tau,dv).
	\end{equation} Further, notice that
	\begin{equation*}\label{ineq:drift}
	\begin{split}
	-&\sum_{\bar{y}\in \mathcal{S}}\langle \ell(\bar{y},\bar{x}),\ell(\bar{z},\bar{x})\rangle Q_{\bar{x},\bar{y}}(\tau,\mu(\tau),u,v)\\ &=
	-\left\langle\sum_{\bar{y}\in \mathcal{S}} \ell(\bar{y},\bar{x})Q_{\bar{x},\bar{y}}(\tau,\mu(\tau),u,v)- f(\tau,\bar{x},u,v),\ell(\bar{z},\bar{x})\right\rangle -\left\langle f(\tau,\bar{x},u,v),\ell(\bar{z},\bar{x})\right\rangle.
	\end{split}
	\end{equation*}
	
	Thanks to (\ref{assumption:ineq:drift}), we obtain that
	\[\left|\left\langle\sum_{\bar{y}\in \mathcal{S}} \ell(\bar{y},\bar{x})Q_{\bar{x},\bar{y}}(\tau,\mu(\tau),u,v)- f(\tau,\bar{x},u,v),\ell(\bar{z},\bar{x})\right\rangle\right|\leq \varepsilon|\ell(\bar{z},\bar{x})|,\] while (see (\ref{intro:R}))
	\[\left|\langle f(\tau,\bar{x},u,v),\ell(\bar{z},\bar{x})\rangle\right|\leq R|\ell(\bar{z},\bar{x})|.\]
	This,  (\ref{ineq:Lambda_Q}), (\ref{ineq:violence_y_squared}), (\ref{equality:Q}) and the inequality  $|\ell(\bar{z},\bar{x})|\leq \sqrt{q_{\bar{z}}(\bar{x})}$ give the following estimate:
	\begin{equation}\label{ineq:E_q}
	\mathbb{E}\Lambda^{\mathcal{Q}}_\tau q_{\bar{z}}(X(\tau))\leq \varepsilon^2+2(\varepsilon+R)[\mathbb{E}q_{\bar{z}}(X(\tau))]^{1/2}.
	\end{equation}
	
	Taking into account the assumption that $\varepsilon\leq 1$, the inequality $0\leq q_{\bar{z}}(x)\leq d$ for every $x\in\td$ and plugging (\ref{ineq:E_q}) into (\ref{equality:expect_q}), we conclude that
	$$\mathbb{E}q_{\bar{z}}(X(t))\leq R_1^2(t-s). $$  Further, estimating the right-hand side of (\ref{ineq:E_q}) using this inequality, we deduce that
	\[\mathbb{E}\Lambda^{\mathcal{Q}}_\tau q_{\bar{z}}(X(\tau))\leq \varepsilon^2+2(\varepsilon+R)R_1\sqrt{\tau-s}.\]
	Plugging this into (\ref{equality:expect_q}), we obtain that
	\[\mathbb{E}q_{\bar{z}}(X(t))\leq \varepsilon^2 (t-s)+\frac{4}{3}(1+R)R_1(t-s)^{3/2}.\] 
	This gives the conclusion of the Lemma.
\end{proof}
\begin{corollary}\label{corollary:distance} Let $\mu(\cdot)$ satisfy (\ref{eq:backward_Kolmogorov}). Then
	$$W_2^2(\widetilde{\mu(t)},\widetilde{\mu(s)})\leq R_1^2(t-s). $$
\end{corollary}
\begin{proof} First notice  \cite{Kol_markov} that one can find a filtered probability space  $(\Omega,\mathcal{F},\{\mathcal{F}_t\}_{t\in [s,r]},P)$ and stochastic processes  $X^{\bar{z}}$, $\bar{z}\in\mathcal{S}$, $X$ those are  defined on $(\Omega,\mathcal{F},\{\mathcal{F}_t\}_{t\in [s,r]},P)$ such that
	\begin{itemize}
\item $(\Omega,\mathcal{F},\{\mathcal{F}_t\}_{t\in [s,r]},P,X^{\bar{z}})$, $\bar{z}\in\mathcal{S}$, and $(\Omega,\mathcal{F},\{\mathcal{F}_t\}_{t\in [s,r]},P,X)$ are generated by  the Kolmogorov matrix $\mathcal{Q}(t,\mu(t),\gamma_{\mathcal{S}}(t),\vartheta_{\mathcal{S}}(t))$;
\item  $X^{\bar{z}}(s)=\bar{z}$, $P$-a.s., 
\item $P(X(s)=\bar{x})=\mu_{\bar{x}}(s)$.\end{itemize} This, in particular, means that $\operatorname{Law}(X(t))=\mu(t)$. Furthermore, for every $\phi\in C(\mathcal{S})$, $$\mathbb{E}\phi(X(t))=\sum_{\bar{z}\in\mathcal{S}}\mu_{\bar{z}}(s)\phi(X^{\bar{z}})(t). $$
	Thus, we have that
	\[ W_2^2(\widetilde{\mu(t)},\widetilde{\mu(s)})\leq \sum_{\bar{z}\in\mathcal{S}}\mathbb{E}\|X^{\bar{z}}(t)-\bar{z}\|^2\mu_{\bar{z}}(s).
	\] From Lemma~\ref{lm:estimate} we obtain that
	\[ W_2^2(\widetilde{\mu(t)},\widetilde{\mu(s)})\leq R_1^2(t-s).\]
\end{proof}

\section{Construction of suboptimal strategy}\label{sect:construction_suboptimal}
In this section, we define the strategy $\mathbb{u}^\Delta$ desired in Theorem~\ref{th:existence_strategy}. The construction  relies on the adaptation of Krasovskii-Subbotin extremal shift rule.

We will consider only the first statement of Theorem~\ref{th:existence_strategy}. The proof of the second statement can be obtained by interchanging of the players. 

Let $\varphi^+$ be a supersolution of (\ref{eq:HJ}). There are several equivalent definitions of the supersolution. We will use one involving the control theory. Recall \cite[\S 12]{Subb_book} that a lower semicontinuous function $\varphi^+$ is a supersolution of (\ref{eq:HJ}) provided that
\begin{enumerate}[label=(S\arabic*)]
	\item\label{def:part:supersolution_boundary} $\varphi^+(T,\mu)\geq \hat{g}(\mu)=g(\widetilde{\mu})$;
	\item\label{def:part:supersolution_viability}
for every $s,r\in [0,T]$, $s<r$, $\mu_*\in  \Sigma$, $\vartheta_{\mathcal{S}}\in \mathcal{V}_{\text{instant}}^{\mathcal{S}}$, there exists a control $[s,r]\ni t\mapsto \gamma_{\mathcal{S}}(t)\in\mathcal{U}_{\text{instant}}^{\mathcal{S}}$ such that, if $\mu(\cdot)$ solves
\begin{equation}\label{eq:mu_control_s_r}
\frac{d}{dt}\mu(t)=\mu(t)\mathcal{Q}(t,\mu(t),\gamma_{\mathcal{S}}(t),\vartheta_{\mathcal{S}}),\ \ \mu(s)=\mu_*,
\end{equation}
then the following inequality holds: \begin{equation}\label{ineq:varphi_plus_s_r}
\varphi^+(r,\mu(r))\leq \varphi^+(s,\mu(s)). 
\end{equation}
\end{enumerate}
\begin{remark}
 Notice that the original definition \cite[\S 12]{Subb_book} requires the convexification of the right-hand side of (\ref{eq:mu_control_s_r}) with respect to the control of the first player. However, we avoid it due to the fact that  the function \[\gamma_{\mathcal{S}}\mapsto \mu\mathcal{Q}(t,\mu,\gamma_{\mathcal{S}},\vartheta_{\mathcal{S}})=\left(\int_U\int_V\sum_{\bar{x}\in \mathcal{S}}\mu_{\bar{x}}{Q}_{\bar{x},\bar{y}}(t,\mu,u,v)\gamma_{\bar{x}}(du)\vartheta_{\bar{x}}(dv)\right)_{\bar{y}\in\mathcal{S}}\] is convex w.r.t. $\gamma_{\mathcal{S}}$.
\end{remark}

Furthermore, let $\hat{u}$ and $\hat{v}$ be measurable functions defined on $[0,T]\times\td\times\td\times \ptd$ and taking values  in $U$ and $V$ respectively satisfying the following condition: for  $t\in [0,T]$, $m\in\ptd$, $x\in\td$, $z\in \rd$, 
\begin{equation}\label{intro:u_hat}
\hat{u}(t,x,y,m)\in \underset{u\in U}{\operatorname{Argmin}}\max_{v\in V}\langle \ell(x,y),f(t,x,m,u,v)\rangle,
\end{equation}
\begin{equation}\label{intro:v_hat}
\hat{v}(t,x,y,m)\in \underset{v\in V}{\operatorname{Argmax}}\min_{u\in U}\langle \ell(x,y),f(t,x,m,u,v)\rangle.
\end{equation} Existence of these functions directly follows from the measurable maximum theorem \cite[18.19]{Infinite_dimensional_analysis}.

Let $t_0$ be an initial time, $\Delta=\{t_k\}_{k=0}^N$ be a partition of $[t_0,T]$. The strategy $\mathbb{u}^\Delta=\{\mathfrak{u}^\Delta_k\}_{k=0}^{N-1}$ is defined  based on the auxiliary motions $\eta(\cdot)$, $\mu(\cdot)$  those serve as guides for the original system governed by (\ref{sys:agent_mfdg}). We define the controls $\mathbb{u}^\Delta_k$, $k=0,\ldots,N-1$ and the motions $\eta(\cdot)$, $\mu(\cdot)$ on intervals $[t_{k-1},t_{k}]$ inductively.
\begin{enumerate}[label=($\mathbb{u}$\arabic{*})]
	\item\label{strategy_intro:initial} Put $\eta(t_0)=\mu(t_0)\triangleq \operatorname{pr}_{\mathcal{S}}(m_0)$.
	\item\label{strategy_intro:step} If $\eta(t)$, $\mu(t)$, $t\in [t_0,t_k]$ are already defined, $m_0,\ldots,m_k\in\ptd$, then 
	let $\pi_k$ be an optimal plan between $m_k$ and $\widetilde{\eta(t_k)}$. Denote
	\begin{equation}\label{intro:alpha_k_def}
	\alpha_k\triangleq (\operatorname{p}^1,\hat{u}(t_k,\cdot,\cdot,m_k))\sharp\pi_k,
	\end{equation}
	\begin{equation}\label{intro:beta_k_def}
	\beta_k\triangleq (\operatorname{p}^2,\hat{v}(t_k,\cdot,\cdot,m_k))\sharp\pi_k.
	\end{equation}
	Put
	\[\mathbb{u}^\Delta_k[m_0,\ldots,m_k]\triangleq \alpha_k.\] Further, notice that $\beta_k$ is supported on $\mathcal{S}\times V$ and $\operatorname{p}^1\sharp\beta_k=\widetilde{\eta(t_k)}$. Let $\beta_k(dv|\bar{x})$ be disintegration of $\beta_k$ w.r.t. $\widetilde{\eta(t_k)}$. We regard each probability $\beta_k(dv|\bar{x})$ as a constant control of the second player in the finite state mean field type game applied at the state $\bar{x}$. Set
	\begin{equation}\label{intro:zeta_k}
	\vartheta_{k,\mathcal{S}}\triangleq (\vartheta_{k,\bar{x}})_{\bar{x}\in\mathcal{S}},\, \text{with } \vartheta_{k,\bar{x}}(dv)\triangleq \beta_k(dv|\bar{x}). 
	\end{equation} Now let $\mu(t)$ $t\in [t_k,t_{k+1}]$ be such that  (\ref{eq:mu_control_s_r}) and (\ref{ineq:varphi_plus_s_r}) holds for $s=t_k$, $r=t_{k+1}$, $\mu_*=\mu(t_k)$ and $\vartheta_{\mathcal{S}}(t)\equiv\vartheta_{k,\mathcal{S}}$, $t\in [t_k,t_{k+1}]$, and some control of the first player $[t_{k},t_{k+1}]\ni t\mapsto \gamma_{k,\mathcal{S}}(t)\in\mathcal{U}_{\text{instant}}^{\mathcal{S}}$.
	
	For $x_*\in\td$, $\bar{y}_*\in \mathcal{S}$, denote by
	$\vartheta_{k,\mathcal{S}}^{x_*,\bar{y}_*}=(\vartheta_{k,\bar{x}}^{x_*,\bar{y}_*})_{\bar{x}\in\mathcal{S}}$ the constant control of the second player with
	\begin{equation}\label{intro:zeta_k_x_bar_y_bar_star}
	\vartheta_{k,\bar{x}}^{x_*,\bar{y}_*}=\left\{\begin{array}{lc}
	\vartheta_{k,\bar{x}}, & \bar{x}\neq \bar{y}_*\\
	\delta_{\hat{v}(t_k,x_*,\bar{y}_*,m_k)}, & \bar{x}=\bar{y}_*.
	\end{array}\right.
	\end{equation}
	
	Let $\eta^{\Delta,x_*,\bar{y}_*}_k$ solve the differential equation
	\begin{equation}\label{eq:eta_x_star_bar_y_star}
	\frac{d}{dt}\eta^{x_*,\bar{y}_*}_k(t)= \eta^{x_*,\bar{y}_*}_k(t)\mathcal{Q}(t,\mu(t),\gamma_{k,\mathcal{S}}(t),\vartheta_{k,\mathcal{S}}^{x_*,\bar{y}_*}),\ \ \eta^{x_*,\bar{y}_*}_k(t_k)=\mathbb{1}_{\bar{y}_*},
	\end{equation} where $\mathbb{1}_{\bar{y}_*}$ stands for the distribution on $\{1,\ldots,|\mathcal{S}|\}$ concentrated at $\bar{y}_*$ (see (\ref{intro:1})) Finally, put, for $t\in [t_{k},t_{k+1}]$, \begin{equation}\label{intro:eta_bar_y_star}
	\eta^{\bar{y}_*}_k(t)\triangleq \int_{\td}\eta^{x_*,\bar{y}_*}_k(t)\pi(dx_*|\bar{y}_*),
	\end{equation}
	\begin{equation}\label{intro:eta_final}
	\eta(t)\triangleq \sum_{\bar{y}_*\in \mathcal{S}}\eta_{\bar{y}_*}(t_k)\eta^{\bar{y}_*}(t).
	\end{equation}

\end{enumerate} 

In the next section, we show that the strategy $\mathbb{u}^\Delta$ defined in this manner provides the estimates desired in Theorem~\ref{th:existence_strategy}. This is performed using a modification of the extremal shift rule.

\section{Extremal shift rule}\label{sect:extremal_shift}
Let $\mathbb{u}^\Delta$ be introduced by rules~\ref{strategy_intro:initial},~\ref{strategy_intro:step} and let $m(\cdot)\in \mathcal{X}_1(t_0,m_0,\Delta,\mathbb{u}^\Delta)$. This means that on each interval there exists a distribution of players' controls $\varkappa_k\in \mathcal{D}_1[\alpha_k]$ (here $\alpha_k$ is given by (\ref{intro:alpha_k_def})) for $m_k=m(t_k)$ such that
\[m(t)=m(t,t_k,m(t_{k}),\varkappa_k),\ \ t\in [t_k,t_{k+1}].\] Recall that the design of strategy $\mathbb{u}^\Delta$ involves auxiliary motions $\mu(\cdot)$ and $\eta(\cdot)$ of the Markov chain those satisfies the initial condition $\mu(t_0)=\eta(t_0)=\operatorname{pr}_{\mathcal{S}}(m_0)$ and are defined on the intervals $[t_k,t_{k+1}]$ by (\ref{intro:zeta_k})--(\ref{intro:eta_final}) (see~\ref{strategy_intro:initial},~\ref{strategy_intro:step}). 

To simplify designation, put
\begin{equation}\label{intro:u_hat_k}
u^\natural_k(x,y)\triangleq \hat{u}(t_k,x,y,m(t_k)), 
\end{equation}
\begin{equation}\label{intro:v_hat_k}
v^\natural_k(x,y)\triangleq \hat{v}(t_k,x,y,m(t_k)). 
\end{equation} Furthermore, for $\zeta\in\mathcal{V}$, set
\begin{equation*}
X_k^{x_*,\bar{y}_*}(t,\zeta)\triangleq x(t,t_k,x_*,m(t),u^\natural_k(x_*,\bar{y}^*),\zeta).
\end{equation*} 
Recall that the probability $\eta(t)$ for $t\in [t_k,t_{k+1}]$ is the averaging of the probabilities $\eta_k^{x_*,\bar{y}_*}(t)$ those are defined by (\ref{eq:eta_x_star_bar_y_star}). The squared 2-Wasserstein distance between $\delta_{X^{x_*,\bar{y}_*}(t_{k+1},\zeta)}$ and $\widetilde{\eta_k^{x_*,\bar{y}_*}(t_{k+1})}$ is equal to
\[W_2^2(\delta_{X^{x_*,\bar{y}_*}_k(t_{k+1},\zeta)},\widetilde{\eta_k^{x_*,\bar{y}_*}(t_{k+1})}) = \sum_{\bar{x}\in\mathcal{S}}\|X^{x_*,\bar{y}_*}_k(t_k,\zeta)-\bar{x}\|^2\eta_{k,\bar{x}}^{x_*,\bar{y}_*}(t_{k+1}),\] when the squared distance between $m(t_{k+1})$ and $\widetilde{\eta(t_{k+1})}$ is estimated as follows:
\begin{equation}\label{ineq:distance_X_x_star_y_bar_eta_t_k_plus}
\begin{split}
W_2&(m(t_{k+1}),\widetilde{\eta(t_{k+1})})\\ &\leq \int_{\td\times\mathcal{S}}\int_{\mathcal{V}} W_2^2(\delta_{X^{x_*,\bar{y}_*}_k(t_{k+1},\zeta)},\widetilde{\eta_k^{x_*,\bar{y}_*}(t_{k+1})})\varkappa_k(d\zeta|x_*,u^\natural_k(x_*,y_*))\pi_k(d(x_*,\bar{y}_*))\\ &=\int_{\td\times\mathcal{S}}\int_{\mathcal{V}} \|X^{x_*,\bar{y}_*}_k(t_{k+1},\zeta)-\bar{x}\|^2\eta_{k,\bar{x}}^{x_*,\bar{y}_*}(t_{k+1})
\varkappa_k(d\zeta|x_*,u^\natural_k(x_*,y_*))\pi_k(d(x_*,\bar{y}_*)),
\end{split}
\end{equation} where $\pi_k$ is the optimal plan between $m(t_k)$ and $\widetilde{\eta(t_k)}$

Below, we estimate $W_2^2(\delta_{X^{x_*,\bar{y}_*}_k(t_{k+1},\zeta)},\widetilde{\eta_k^{x_*,\bar{y}_*}(t_{k+1})})$ by some function of $\|x_*-\bar{y}_*\|^2$, $W_2^2(m(t_k),\widetilde{\eta(t_k)})$ and $W_2^2(m(t_k),\widetilde{\mu(t_k)})$ (Lemma~\ref{lm:extremal_shift_two_point}). Then we evaluate the  distance between $\mu(t_{k+1})$ and $\eta(t_{k+1})$. Based on this we derive the estimate of squared Wasserstein distance between $m(t_{k+1})$ and $\widetilde{\eta(t_{k+1})})$ by $W_2^2(m(t_k),\widetilde{\eta(t_k)})$. This will give the proof of the main theorem.
 
To state the Lemma~\ref{lm:extremal_shift_two_point}, denote 
\[\begin{split}
\varsigma_f(\theta)\triangleq \sup\{\|f(&t',x,m,u,v)-f(t'',x,m,u,v)\|:\\&t',t''\in [0,T],x\in\td,m\in\ptd,u\in U,v\in V,\ \ |t'-t''|\leq \theta\}.\end{split}\] Since $f$ is continuous on the compact $[0,T]\times\td\times\ptd\times U\times V$, the function $\varsigma_f$ is such that $\varsigma_f(\theta)\rightarrow 0$ as $\theta\rightarrow 0$.

\begin{lemma}\label{lm:extremal_shift_two_point} For any $\zeta\in\mathcal{V}$, the following estimate holds:
\begin{equation}\label{ineq:wasserstein_delta_nu}
\begin{split}
W_2^2(\delta_{X_k^{x_*,\bar{y}_*}(t_{k+1})},\widetilde{\eta^{x_*,\bar{y}_*}_k(t_{k+1})})\leq \|x_*-\bar{y}_*\|^2(&1+(3L+1)(t_{k+1}-t_k))\\+L W_2^2(m(t_{k}),\widetilde{\mu(t_k)})(t_{k+1}-t_k)&+2\varepsilon^2 (t_{k+1}-t_k)\\&+\varsigma_2(t_{k+1}-t_k)\cdot(t_{k+1}-t_k),
\end{split}
\end{equation}
where $\varsigma_2(\cdot)$ is  determined only by $f$, $Q$, $\mathcal{S}$ and satisfies  $\varsigma_2(\theta)\rightarrow 0$ as $\theta\rightarrow 0$.
\end{lemma}
\begin{proof}
	Let $(\Omega,\mathcal{F},\{\mathcal{F}_t\}_{t\in [t_{k},t_{k+1}]},Y,P)$ be a Markov chain generated by the Kolmogorov matrix $\mathcal{Q}(t,\mu(t),\gamma_{k,\mathcal{S}}(t),\vartheta_{k,\mathcal{S}}^{x_*,\bar{y}_*})$ and the initial distribution  at the time $s$ equal to $\mathbb{1}_{\bar{y}_*}$. Notice that
	$P(Y(t)=\bar{x})=\eta_{k,\bar{x}}^{x_*,\bar{y}_*}(t)$.
	
	We shall prove that	
	\begin{equation}\label{ineq:expectation_X_Y}
	\begin{split}
	\mathbb{E}\|X_k^{x_*,\bar{y}_*}(t_{k+1})-Y(t_{k+1})\|^2\leq \|x_*-\bar{y}_*\|^2(1&+(3L+1)(t_{k+1}-t_k))\\+L W_2^2(m(t_k),\widetilde{\mu(t_k)})&(t_{k+1}-t_k)+2\varepsilon^2 (t_{k+1}-t_k)\\&+\varsigma_2(t_{k+1}-t_k)\cdot(t_{k+1}-t_k)
	\end{split}
	\end{equation} for some function $\varsigma_2$ such that $\varsigma_2(\theta)\rightarrow 0$ as $\theta\rightarrow 0$. This will imply (\ref{ineq:wasserstein_delta_nu}).
	
	
	For shortness, set
	\begin{equation}\label{intro:u_star_def}
	u_*\triangleq u^\natural_k(x_*,\bar{y}_*)=\hat{u}(t_k,x_*,\bar{y}_*,m(t_k)),
	\end{equation}
	\begin{equation}\label{intro:v_star_def}
	v_*\triangleq v^\natural_k(x_*,\bar{y}_*)=\hat{v}(t_k,x_*,\bar{y}_*,m(t_k)).
	\end{equation}
	Recall that $X^{x_*,\bar{y}_*}_k(\cdot,\zeta)$ satisfies
	\[\frac{d}{dt}X^{x_*,\bar{y}_*}_k(t)=\int_V f(t,X^{x_*,\bar{y}_*}_k(t),m(t),u_*,v)\zeta(dv|t),\ \ x(t_k)=x_*.\] If $x'\in\rd$, then, with some abuse of notation, we write $f(t,x',m,u,v)$ instead $f(t,\{x'\}+\mathbb{Z}^d,m,u,v)$.
	Now, let $x'_*\in x_*$. Consider the solution of the following differential equation in $\rd$:
	\begin{equation}\label{eq:f_rd}
	\frac{d}{dt}x'(t)=\int_Vf(t,x'(t),m(t),u_*,v)\zeta(dv|t),\ \ x'(t_k)=x_*'.
	\end{equation} Notice that $x'(t)\in X^{x_*,\bar{y}_*}_k(t)$.
	
	For any $Y'_{t_k}\in Y(t_k)$, $Y'_{t_{k+1}}\in Y(t_{k+1})$, we have that
	\[\begin{split}
	\|x(t_{k+1})-Y(t_{k+1})&\|^2\leq\|x'(t_{k+1})-Y'_{t_{k+1}}\|^2\\=\|(x&{}'(t_{k+1})-x'(t_k))+(x'(t_k)-Y'_{t_k})-(Y'_{t_{k+1}}-Y'_{t_k})\|^2\\ =\|x'&(t_k)-Y'_{t_k}\|^2+\|Y'_{t_{k+1}}-Y'_{t_k}\|^2+\|x'(t_{k+1})-x'(t_k)\|^2
	\\&-2\langle x'(t_{k+1})-x'(t_k),Y'_{t_{k+1}}-Y'_{t_k}\rangle\\&+2\langle x'(t_{k+1})-x'(t_k),x'(t_k)-Y'_{t_k}\rangle-2\langle Y'_{t_{k+1}}-Y'_{t_k},x'(t_k)-Y'_{t_k}\rangle\\ \leq
	\|x'&(t_k)-Y'_{t_k}\|^2+2\|Y'_{t_{k+1}}-Y'_{t_k}\|^2+2\|x'(t_{k+1})-x'(t_k)\|^2
	\\&+2\langle x'(t_{k+1})-x'(t_k),x'(t_k)-Y'_{t_k}\rangle-2\langle Y'_{t_{k+1}}-Y'_{t_k},x'(t_k)-Y'_{t_k}\rangle.
		\end{split}
	\]
	
	Now choose $Y'_{t_k}\in Y(t_k)$ $Y'_{t_{k+1}}\in Y(t_{k+1})$ such that 
	\[x'(t_k)-Y'_{t_k}=\ell(x_*,Y(t_k)),\, Y'_{t_{k+1}}-Y'_{t_k}=\ell(Y(t_{k+1}),Y(t_k)).\] Therefore,
	\begin{equation}\label{lemma_extremal:ineq:x_r_Y_r}
	\begin{split}
	\|x(t_{k+1})&-Y(t_{k+1})\|^2\\\leq \|x&(t_k)-Y(t_k)\|^2+2\|Y(t_{k+1})-Y(t_k)\|^2+2\|x'(t_{k+1})-x'(t_k)\|^2
	\\&+2\langle x'(t_{k+1})-x'(t_k),\ell(x(t_k),Y(t_k))\rangle\\&-2\langle \ell(Y(t_{k+1}),Y(t_k)),\ell(x(t_k),Y(t_k))\rangle.
	\end{split}
	\end{equation}
	
	We have that 
	\begin{equation}\label{lemma_extremal:ineq:x_prime_squared}
	\|x'(t_{k+1})-x'(t_k)\|^2\leq R(t-s)^2.
	\end{equation} Lemma~\ref{lm:estimate} implies that
	\begin{equation}\label{lemma_extremal:ineq:expectation_Y_r_s}
	\mathbb{E}\|Y(t_{k+1})-Y(t_k)\|^2\leq \varepsilon^2(t-s)+\varsigma_1(t-s)\cdot (t-s).
	\end{equation}
	
	Now, let us consider the term $\langle x'(t_{k+1})-x'(t_k),\ell(x(t_k),Y(t_k))$. To evaluate it, recall that $x'(\cdot)$ satisfies  differential equation (\ref{eq:f_rd}). 
	We have that
	\begin{equation}\label{equality:x_r_f}
	\begin{split}
	\langle x'(t_{k+1})-x'(&t_k),\ell(x_*,Y(t_k))\rangle\\&=\left\langle\int_s^t\int_V  f(\tau,x'(\tau),m(\tau),u_*,v)\zeta(dv|\tau)d\tau,\ell(x_*,Y(t_k))\right\rangle.
	\end{split}
	\end{equation}
	Further, the continuity of $f$  and the Lipschitz continuity with respect to $x$ and $m$ yields
	\[\begin{split}
	\|f(\tau,x'(\tau),&m(\tau),u_*,v)-f(t_k,x_*,m(t),u_*,v)\|\\&\leq \varsigma_f(t_{k+1}-t_k)+L\|x'(\tau)-x'(t_k)\|+LW_2(m(\tau),m(t_k)),	\end{split}
	\] when $t\in [t_k,t_{k+1}]$. This and inequalities  $\|x'(\tau)-x'(t_k)\|\leq R|\tau-t_k|$ (see (\ref{intro:R})), $W_2(m(\tau),m(t_k))\leq R|\tau-t_k|$ (see Lemma~\ref{lm:deterministic_motion_estimate}) imply
\[
\|f(\tau,x'(\tau),m(\tau),u_*,v)-f(t_k,x_*,m(t),u_*,v)\|\leq \varsigma_f(\tau-t_k)+2LR(\tau-t_k).	
\] Plugging this estimate into (\ref{equality:x_r_f}), we get the following:
\begin{equation}\label{lemma_extremal:ineq:x_r_f_in_s}
\begin{split}
\Bigl|\langle x'(t_{k+1})-x'(t_k),\ell(x_*,Y(t_k))&\rangle\\-\Bigl\langle\int_s^t\int_V  f(t_k,x_*,&m(t_k),u_*,v)\zeta(dv|\tau)d\tau,\ell(x_*,Y(t_k))\Bigr\rangle\Bigr|\\\leq \sqrt{d} [\varsigma_f(t_{k+1}-t_k)+LR(&t_k-t_{k+1})]\cdot (t_{k+1}-t_k).
\end{split}
\end{equation}
	
Next, we evaluate the term $\mathbb{E}\langle \ell(Y(t_{k+1}),Y(t_k)),\ell(x_*,Y(t_k))\rangle$. Recall that $Y(t_k)=\bar{y}_*$ $P$-a.s. Furthermore,  the function $t\mapsto\eta^{x_*,\bar{y}_*}_{k,\mathcal{S}}(t)=(\eta^{x_*,\bar{y}_*}_{k,\bar{x}}(t))_{\bar{x}\in\mathcal{S}}$ is such that $\eta^{x_*,\bar{y}_*}_{k,\bar{x}}(t)\triangleq P(Y(\tau)=\bar{x})$ and  satisfies ODE (\ref{eq:nu}) with initial condition  $\nu_{*}=\mathbb{1}_{\bar{y}_*}$. Thus, since $Q(t,\mu,u,v)$ is defined on the compact space, we have that
\begin{equation}\label{ineq:nu_nu_star}
\|\eta^{x_*,\bar{y}_*}_k(t)-\mathbb{1}_{\bar{y}_*}\|_2\leq C_3(t-t_k),
\end{equation} where $C_3$ is a constant dependent only on $Q$.

Using the fact that  the process (\ref{property:phi_martinglate}) is the $\{\mathcal{F}_t\}_{t\in[s,r]}$-martingale  for the function $\phi(x)\triangleq \langle \ell(x,\bar{y}_*),\ell(x_*,\bar{y}_*)\rangle$ when the controls $\gamma_{k,\mathcal{S}}$ and $\vartheta_{k,\mathcal{S}}^{x_*,\bar{y}_*}$ are used, we conclude that 
\[\begin{split}
 \mathbb{E}\langle \ell(Y(t_{k+1})&,Y(t_k)),\ell(x_*,Y(t_k))\rangle \\&=\mathbb{E}\int_{t_k}^{t_{k+1}} \Lambda^{\mathcal{Q}}_\tau[\mu(\tau),\gamma_{k,\mathcal{S}}(\tau),\vartheta_{k,\mathcal{S}}^{x_*,\bar{y}_*}]\Bigl \langle\ell(\cdot,\bar{y}_*),\ell(x_*,\bar{y}_*)\Bigr\rangle (Y(\tau)) d\tau.\end{split}
\] The definition of the generator $\Lambda_t^\mathcal{Q}$ (see (\ref{intro:generator_Markov})) implies that
\[ \begin{split}
\Lambda^{\mathcal{Q}}_\tau&[\mu(\tau),\gamma_{k,\mathcal{S}}(\tau),\vartheta_{k,\mathcal{S}}^{x_*,\bar{y}_*}]\langle\ell(\cdot,\bar{y}_*),\ell(x_*,\bar{y}_*)\rangle(\bar{x})\\&=
\left\langle\sum_{\bar{y}\in\mathcal{S}}(\ell(\bar{y},\bar{y}_*)-\ell(\bar{x},\bar{y}_*))\mathcal{Q}_{\bar{x},\bar{y}}(\tau,\mu(\tau),\gamma_{k,\mathcal{S}}(\tau),\vartheta_{k,\mathcal{S}}^{x_*,\bar{y}_*}),\ell(x_*,\bar{y}_*)\right\rangle
\end{split}
\] Hence,
\begin{equation}\label{equality:Expectation_ell}
\begin{split}
\mathbb{E}&\langle \ell(Y(t_{k+1}),Y(t_k)),\ell(x_*,Y(t_k))\rangle\\&=\int_{t_k}^{t_{k+1}}\Bigl\langle\sum_{\bar{x}\in \mathcal{S}}\eta^{x_*,\bar{y}_*}_{k,\bar{x}}(\tau)\sum_{\bar{y}\in\mathcal{S}}(\ell(\bar{y},\bar{y}_*)-\ell(\bar{x},\bar{y}_*))\mathcal{Q}_{\bar{x},\bar{y}}(\tau,\mu(\tau),\gamma_{k,\mathcal{S}}(\tau),\vartheta_{k,\mathcal{S}}^{x_*,\bar{y}_*}),\ell(x_*,\bar{y}_*)\Bigr\rangle.
\end{split}
\end{equation}

Notice that, from (\ref{assumption:ineq:drift}) and the fact that $\vartheta_{k,\bar{y}_*}^{x_*,\bar{y}_*}=\delta_{v_*}$, it follows, that, if $\bar{x}=\bar{y}_*$, 
\begin{equation}\label{ineq:dynamic_Markov_f}
\begin{split}
\Bigl|\sum_{\bar{y}\in\mathcal{S}}(\ell(\bar{y},\bar{y}_*)-\ell&(\bar{x},\bar{y}_*))\mathcal{Q}_{\bar{x},\bar{y}}(\tau,\mu(\tau),\gamma_{k,\mathcal{S}}(\tau),\vartheta_{k,\mathcal{S}}^{x_*,\bar{y}_*})\\&-\int_Uf(\tau,\bar{y}_*,\widetilde{\mu(\tau)},u,v_*)\gamma_{k,\bar{y}_*}(\tau,du)\Bigr|\leq \varepsilon.
\end{split}
\end{equation}

Furthermore, (\ref{ineq:nu_nu_star}) gives that $|\eta^{x_*,\bar{y}_*}_{k,\bar{y}_*}(\tau)-1|\leq C_3(\tau-s).$

If $\bar{x}\neq \bar{y}_*$, then, as above, $\eta^{x_*,\bar{y}_*}_{k,\bar{x}}(\tau)\leq C_3(\tau-s)$. Additionally,  $\|\ell(x,y)\|\leq \sqrt{d}$ on $\td\times\td$, when the functions $Q_{\bar{x},\bar{y}}$ are bounded. Thus,
\[
\begin{split}
\Bigl|\sum_{\bar{x}\in\mathcal{S}}\eta^{x_*,\bar{y}_*}_{k,\bar{x}}(\tau)\sum_{\bar{y}\in \mathcal{S}} (\ell(\bar{y},\bar{y}_*)-\ell(&\bar{x},\bar{y}_*))\mathcal{Q}_{\bar{x},\bar{y}}(\tau,\mu(\tau), \gamma_{k,\bar{x}}(\tau),\vartheta_{k,\bar{x}}^{x_*,\bar{y}_*})\\ &-\sum_{\bar{y}\in \mathcal{S}} \ell(\bar{y},\bar{y}_*)\mathcal{Q}_{\bar{y}_*,\bar{y}}(\tau,\mu(\tau),\gamma_{k,\bar{y}_*}(\tau), \vartheta_{k,\bar{y}_*}^{x_*,\bar{y}_*})\Bigr| \\ \leq 
|\eta^{x_*,\bar{y}_*}_{k,\bar{y}_*}(\tau)-1|\sum_{\bar{y}\in \mathcal{S}}|&\ell(\bar{y},\bar{y}_*)|\cdot |\mathcal{Q}_{\bar{y}_*,\bar{y}}(\tau,\mu(\tau),\gamma_{k,\bar{y}_*}(\tau),\vartheta_{k,\bar{y}_*}^{x_*,\bar{y}_*})|\\+
\sum_{\bar{x}\neq\bar{y}_*}|\eta^{x_*,\bar{y}_*}_{k,\bar{x}(\tau)}|\sum_{\bar{y}\in \mathcal{S}} &|\ell(\bar{y},\bar{y}_*)-\ell(\bar{x},\bar{y}_*)|\cdot |\mathcal{Q}_{\bar{x},\bar{y}}(\tau,\mu(\tau), \gamma_{k,\bar{x}}(\tau),\vartheta_{k,\bar{x}}^{x_*,\bar{y}_*})|.
\end{split}
\] The right-hand side of this inequality can be estimated by $C_4|\tau-s|$, where $C_4$ is a constant determined by the matrix $Q$. This and (\ref{ineq:dynamic_Markov_f}) yield that
\[
\begin{split}
\Bigl|\Bigl\langle\sum_{\bar{x}\in \mathcal{S}}\eta^{x_*,\bar{y}_*}_{k,\bar{x}}&(\tau)\sum_{\bar{y}\in\mathcal{S}}(\ell(\bar{y},\bar{y}_*)-\ell(\bar{x},\bar{y}_*))\mathcal{Q}_{\bar{x},\bar{y}}(\tau,\mu(\tau),\gamma_{k,\bar{x}}(\tau),\vartheta_{k,\bar{x}}^{x_*,\bar{y}_*}),\ell(x_*,\bar{y}_*)\Bigr\rangle\\ -\Bigl\langle  &\int_Uf(\tau,\bar{y}_*,\widetilde{\mu(\tau)},u,v_*)\gamma_{k,\bar{y}_*}(\tau,du),\ell(x_*,\bar{y}_*) \Bigr\rangle\Bigr|\leq \varepsilon|\ell(x_*,\bar{y}_*)|+C_4\sqrt{d}(\tau-t_k).
\end{split}
\] Recall that we denote the modulus of continuity of $f$ w.r.t. time variable by $\varsigma_f$. Furthermore, $W_2(\widetilde{\mu(t_k)},\widetilde{\mu(\tau)})\leq R_1(\tau-t_k)^{1/2}$. Hence, 
\[
\begin{split}
\Bigl|\Bigl\langle\sum_{\bar{x}\in \mathcal{S}}\eta^{x_*,\bar{y}_*}_{k,\bar{x}}&(\tau)\sum_{\bar{y}\in\mathcal{S}}(\ell(\bar{y},\bar{y}_*)-\ell(\bar{x},\bar{y}_*))\mathcal{Q}_{\bar{x},\bar{y}}(\tau,\mu(\tau),\gamma_{k,\bar{x}}(\tau),\vartheta_{k,\bar{x}}^{x_*,\bar{y}_*}),\ell(x_*,\bar{y}_*)\Bigr\rangle\\ -\Bigl\langle  &\int_Uf(t_k,\bar{y}_*,\widetilde{\mu(t_k)},u,v_*)\gamma_{k,\bar{y}_*}(\tau,du),\ell(x_*,\bar{y}_*) \Bigr\rangle\Bigr|\leq \varepsilon|\ell(x_*,\bar{y}_*)|+\varsigma'_2(\tau-t_k).
\end{split}
\] Here $\varsigma'_2(\theta)\triangleq C_4\sqrt{d}\theta+\varsigma_f(\theta)+LR_1\theta^{1/2}$ is  such that $\varsigma'_2(\theta)\rightarrow 0$ as $\theta\rightarrow 0$.

Combining the last estimate with (\ref{equality:Expectation_ell}), we arrive at the following estimate:
\begin{equation}\label{lemma_extremal:ineq:expectation_Y_r_Y_s_f}
\begin{split}
\Bigl|\mathbb{E}\langle \ell(Y(t_{k+1}),Y(t_k)),&\ell(x_*,Y(t_k))\rangle\\ -&\int_{t_k}^{t_{k+1}}\Bigl\langle  \int_Uf(t_k,\bar{y}_*,\widetilde{\mu(t_k)},u,v_*)\gamma_{k,\bar{y}_*}(\tau,du)d\tau,\ell(x_*,\bar{y}_*) \Bigr\rangle\Bigr|\\\leq \frac{\varepsilon^2}{2}&(t_{k+1}-t_k)+\frac{\|x_*-\bar{y}_*\|^2}{2}(t_{k+1}-t_k)+\varsigma'_2(t_{k+1}-t_k)(t_{k+1}-t_k).
\end{split}
\end{equation}

Now, turn back to inequality (\ref{lemma_extremal:ineq:x_r_Y_r}). Expectation of its right-hand side is evaluated in~(\ref{lemma_extremal:ineq:x_prime_squared}),~(\ref{lemma_extremal:ineq:expectation_Y_r_s}),~(\ref{lemma_extremal:ineq:x_r_f_in_s}),~(\ref{lemma_extremal:ineq:expectation_Y_r_Y_s_f}). They imply the following estimate:
\begin{equation}\label{lemma_extremal:ineq:main_ineq:pre_lipschitz}
\begin{split}
\mathbb{E}\|x(t_{k+1})-Y(t_{k+1})\|^2\leq &\|x_*-y_*\|+2\varepsilon^2(t_{k+1}-t_k)+\|x_*-\bar{y}_*\|^2(t_{k+1}-t_k)\\& +2\int_{t_k}^{t_{k+1}}\int_V  \langle f(t_k,x_*,m(t_k),u_*,v),\ell(x_*,\bar{y}_*)\rangle\zeta(dv|\tau)d\tau\\&-2\int_{t_k}^{t_{k+1}} \int_U\langle f(t_k,\bar{y}_*,\widetilde{\mu(t_k)},u,v_*),\ell(x_*,\bar{y}_*)\rangle\gamma_{k,\bar{y}_*}(\tau,du)d\tau \\&+\varsigma_2(t_{k+1}-t_k)\cdot(t_{k+1}-t_k), 
\end{split}
\end{equation} where 
\begin{equation}\label{intro:varsigma_2}
\varsigma_2(\theta)\triangleq R\theta+\varsigma_1(\theta)+2[\varsigma_f(\theta)+LR\theta]+2\varsigma'_2(\theta).
\end{equation}

Since $f$ is Lipschitz continuous w.r.t. to $x$ and $m$ we have that
\[
\begin{split}
|\langle f(t_k,\bar{y}_*,\widetilde{\mu(t_k)},u,v_*),&\ell(x_*,\bar{y}_*)\rangle-\langle f(t_k,x_*,m(t_k),u,v_*),\ell(x_*,\bar{y}_*)\rangle|\\ &\leq L\|x_*-\bar{y}_*\|^2+LW_2(m(t_k),\widetilde{\mu(t_k)})\|x_*-\bar{y}_*\|\\&\leq \frac{3}{2}L\|x_*-\bar{y}_*\|^2+\frac{1}{2}LW_2^2(m(t_k),\widetilde{\mu(t_k)}).
\end{split}
\]

Using this, we can estimate the right-hand side of (\ref{lemma_extremal:ineq:main_ineq:pre_lipschitz}) as follows:
\begin{equation}\label{lemma_extremal:ineq:main_ineq:after_lipschitz}
\begin{split}
\mathbb{E}\|x(t_{k+1})-&Y(t_{k+1})\|^2\\\leq &\|x_*-\bar{y}_*\|+2\varepsilon^2(t_{k+1}-t_k)+(1+3L)\|x_*-\bar{y}_*\|^2(t_{k+1}-t_k)\\&+LW_2^2(m(t_k),\widetilde{\mu(t_k)})(t_{k+1}-t_k)\\&+2\int_{t_k}^{t_{k+1}}\int_V  \langle f(t_k,x_*,m(t_k),u_*,v),\ell(x_*,\bar{y}_*)\rangle\zeta(dv|\tau)d\tau\\&-2\int_{t_k}^{t_{k+1}} \int_U\langle f(t_k,x_*,m(t_k),u,v_*),\ell(x_*,\bar{y}_*)\rangle\gamma_{k,\bar{y}_*}(\tau,du)d\tau \\&+\varsigma_2(t_{k+1}-t_k)\cdot(t_{k+1}-t_k), 
\end{split}
\end{equation}
The choice of the controls $u_*$ and $v_*$ (see (\ref{intro:u_hat}), (\ref{intro:v_hat}),  (\ref{intro:u_star_def}), (\ref{intro:v_star_def})) gives that, for any $u\in U$, $v\in V$,
\[\langle f(t_k,x_*,m(t_k),u_*,v),\ell(x_*,\bar{y}_*)\rangle-\langle f(t_k,x_*,m(t_k),u,v_*),\ell(x_*,\bar{y}_*)\rangle\leq 0.
\] This and (\ref{lemma_extremal:ineq:main_ineq:after_lipschitz}) imply (\ref{ineq:expectation_X_Y}).  Simultaneously, very definition of the function $\varsigma_2$ implies that $\varsigma_2(\epsilon)\rightarrow 0$ as $\epsilon\rightarrow 0$.

\end{proof}

The following lemma provides the estimate between $\mu(\cdot)$ and $\eta(\cdot)$ defined in~\ref{strategy_intro:step}.
\begin{lemma}\label{lm:mu_eta_estiame} Let $\mu(\cdot)$ solve (\ref{eq:mu_control_s_r}) on each interval $[t_k,t_{k+1}]$ with $\vartheta_{k,\mathcal{S}}$ given by (\ref{intro:zeta_k})  and let $\eta(\cdot)$ be defined by (\ref{intro:eta_final}). Then there exists a constant $C_5$ determined only by $Q$ such that
	\begin{equation*}\label{ineq:mu_eta_t}
	\|\mu(t)-\eta(t)\|_2\leq C_5d(\Delta).
	\end{equation*} 
\end{lemma}
\begin{proof}
	First, consider $\nu(\cdot)$ satisfying ODE (\ref{eq:nu}) on $[s,r]$. We have that \[\nu(t)=\nu_*\exp\left\{\int_{s}^t\mathcal{Q}(\tau,\mu(\tau),\gamma_{\mathcal{S}}(\tau),\vartheta_{\mathcal{S}}(\tau))d\tau.\right\}\] Expanding the exponent and using the fact that  the matrices $Q(t,\mu,u,v)$ (and, thus, $\mathcal{Q}(t,\mu,\gamma_{\mathcal{S}},\vartheta_{\mathcal{S}})$) are uniformly bounded, we conclude that
	\begin{equation}\label{ineq:nu_linear}
	\left\|\nu(r)-\nu(s)- \nu(s)\int_s^r\mathcal{Q}(\tau,\mu(\tau),\gamma_{\mathcal{S}}(\tau),\vartheta_{\mathcal{S}}(\tau))d\tau\right\|_2\leq C'_5(r-s)^2.
	\end{equation} Here $C'_5$ is a constant  (certainly, dependent on $Q$).
	
	Now let $x_*\in\td$, whereas $\bar{y}_*\in \mathcal{S}$. Recall that $\eta^{x_*,\bar{y}_*}_{k}$ solves (\ref{eq:eta_x_star_bar_y_star}). Estimate (\ref{ineq:nu_linear}) implies that
	\[\left\|\eta^{x_*,\bar{y}_*}_{k}(t_{k+1})-\mathbb{1}_{\bar{y}_*}- \mathbb{1}_{\bar{y}_*}\int^{t_{k+1}}_{t_k} \mathcal{Q}(\tau,\mu(\tau),\gamma_{k,\mathcal{S}}(\tau),\vartheta_{k,\mathcal{S}}^{x_*,\bar{y}_*}) d\tau \right\|_2\leq C'_5(t_{k+1}-t_k)^2.\]
	Notice that the very definitions of $\vartheta_{k,\mathcal{S}}$ and  $\vartheta_{k,\mathcal{S}}^{x_*,\bar{y}_*}$  (see (\ref{intro:zeta_k}) and (\ref{intro:zeta_k_x_bar_y_bar_star})) yield that, for every $t\in [0,T]$, $\mu'\in\Sigma$, $\gamma_{\mathcal{S}}\in\mathcal{U}_{\text{instant}}^{\mathcal{S}}$, $\bar{y}_*\in\mathcal{S}$,
	\[\mathcal{Q}(t,\mu',\gamma_{\mathcal{S}},\vartheta_{k,\mathcal{S}})= \int_{\td}\mathcal{Q}(t,\mu',\gamma_{\mathcal{S}},\vartheta_{k,\mathcal{S}}^{x_*,\bar{y}_*})\pi(dx_*|\bar{y}_*).\]   Hence,
	 \begin{equation*}
	 \left\|\eta^{\bar{y}_*}_{k}(t_{k+1})-\mathbb{1}_{\bar{y}_*}- \mathbb{1}_{\bar{y}_*}\int_{t_k}^{t_{k+1}}\mathcal{Q}(\tau,\mu(\tau),\gamma_{k,\mathcal{S}}(\tau),\vartheta_{k,\mathcal{S}})d\tau\right\|_2\leq C'_5(t_{k+1}-t_k)^2.
	 \end{equation*} Here $\eta^{\bar{y}_*}_{k}(\cdot)$ is defined by (\ref{intro:eta_bar_y_star}). 
	Now, let $\mu_k^{\bar{y}_*}(\cdot)$ solve (\ref{eq:nu}) on $[t_k,t_{k+1}]$ with $\mu_k^{\bar{y}_*}(t_k)=\mathbb{1}_{\bar{y}_*}$ and controls $\gamma_{k,\mathcal{S}}$ and $\vartheta_{k,\mathcal{S}}$.
	Due to (\ref{ineq:nu_linear}), we have that
	\[\left\|\mu_k^{\bar{y}_*}(t_{k+1})-\mathbb{1}_{\bar{y}_*}-\mathbb{1}_{\bar{y}_*} \int_{t_k}^{t_{k+1}}\mathcal{Q}(\tau,\mu(\tau),\gamma_{k,\mathcal{S}}(\tau),\vartheta_{k,\mathcal{S}})d\tau\right\|_2\leq C'_5(t_{k+1}-t_k)^2. \] Therefore,
	\begin{equation}\label{ineq:eta_mu_t_k_step}
	\|\eta_k^{\bar{y}_*}(t_{k+1})-\mu_k^{\bar{y}_*}(t_{k+1})\|_2\leq 2 C'_5(t_{k+1}-t_k)^2.
	\end{equation} Now recall that
	\[\eta(t)=\sum_{\bar{y}_*\in \mathcal{S}}\eta_{k,\bar{y}_*}(t_k)\eta_k^{\bar{y}_*}(t),\ \ \mu(t)=\sum_{\bar{y}_*\in \mathcal{S}}\mu_{\bar{y}_*}(t_k)\mu_k^{\bar{y}_*}(t).\]
	Thus,
	\begin{equation}\label{ineq:eta_mu_k_plus_eta_mu_k}
	\begin{split}
	\|\eta(t_{k+1})-\mu(t_{k+1}&)\|_2\\
	\leq \Bigl\|\sum_{\bar{y}_*\in \mathcal{S}}&\eta_{k,\bar{y}_*}(t_k)\eta_k^{\bar{y}_*}(t_{k+1})-\sum_{\bar{y}_*\in \mathcal{S}}\eta_{k,\bar{y}_*}(t_k)\mu_k^{\bar{y}_*}(t_{k+1})\Bigr\|_2\\
	&+ \Bigl\|\sum_{\bar{y}_*\in \mathcal{S}}\eta_{k,\bar{y}_*}(t_k)\mu_k^{\bar{y}_*}(t_{k+1})-\sum_{\bar{y}_*\in \mathcal{S}}\mu_{\bar{y}_*}(t_k)\mu_k^{\bar{y}_*}(t_{k+1})\Bigr\|_2
	\end{split}
	\end{equation} 
	
	The first term in the right-hand side of this inequality can be estimated as follows:
	\[\begin{split}
	\Bigl\|\sum_{\bar{y}_*\in \mathcal{S}}\eta_{k,\bar{y}_*}&(t_k)\eta_k^{\bar{y}_*}(t_{k+1})-\sum_{\bar{y}_*\in \mathcal{S}}\eta_{k,\bar{y}_*}(t_k)\mu_k^{\bar{y}_*}(t_{k+1})\Bigr\|_2 \\ &\leq  \sum_{\bar{y}_*\in\mathcal{S}}\eta_{k,\bar{y}_*}(t_k)\|\eta_k^{\bar{y}_*}(t_{k+1})-\mu_k^{\bar{y}_*}(t_{k+1})\|_2\\ &\leq \sum_{\bar{y}_*\in\mathcal{S}}\eta_{k,\bar{y}_*}(t_k)\cdot 2C_5'(t_{k+1}-t_k)^2=2C_5'(t_{k+1}-t_k)^2.\end{split}\] In the last inequality above we used estimate (\ref{ineq:eta_mu_t_k_step}). 
	
	To evaluate the second term, notice that 
	the functions \[t\mapsto \sum_{\bar{y}_*\in \mathcal{S}}\eta_{k,\bar{y}_*}(t_k)\mu_k^{\bar{y}_*}(t)\text{ and }t\mapsto \sum_{\bar{y}_*\in \mathcal{S}}\mu_{\bar{y}_*}(t_k)\mu_k^{\bar{y}_*}(t)\] satisfy (\ref{eq:nu}) with the initial conditions at the time $t_k$  $\nu_*=\eta(t_k)$ and $\nu_*=\mu(t_k)$ respectively. Thus, there exists a constant $C_5''$ such that
	\[\Bigl\|\sum_{\bar{y}_*\in \mathcal{S}}\eta_{k,\bar{y}_*}(t_k)\mu_k^{\bar{y}_*}(t)- \sum_{\bar{y}_*\in \mathcal{S}}\mu_{\bar{y}_*}(t_k)\mu_k^{\bar{y}_*}(t)\Bigr\|_2\leq \|\eta(t_k)-\mu(t_k)\|_2e^{C_5''(t_{k+1}-t_k)}.\] As above, $C_5''$ is a constant determined only by $Q$. Using (\ref{ineq:eta_mu_k_plus_eta_mu_k}), we conclude that
	\[\|\eta(t_{k+1})-\mu(t_{k+1})\|_2\leq 2C_5'(t_{k+1}-t_k)^2+\|\eta(t_{k})-\mu(t_{k})\|_2e^{C_5''(t-t_k)}.\]
	
	From this and the fact that $\eta(t_0)=\mu(t_0)$ one can derive the estimate
	\[\|\eta(t_{k})-\mu(t_{k})\|_2\leq 2C_5'e^{C_5''(t_k-t_0)}(t_k-t_0)d(\Delta).\] This gives the conclusion of the lemma with $C_5\triangleq 2C_5'e^{C_5''T}T$.
\end{proof}

\begin{lemma}\label{lm:extremal_shift_m_mu}  
The following inequality holds for every $k=0,\ldots,N-1$:
\[\begin{split}
W_2^2(m(t_{k+1}),\widetilde{\eta(t_{k+1})})\leq W_2^2(m(t_k)&,\widetilde{\eta(t_k)})(1+(4L+1)(t_{k+1}-t_k))\\&+2\varepsilon^2(t_{k+1}-t_k)+\varsigma_3(d(\Delta))\cdot(t_{k+1}-t_k).
\end{split}\] Here $\varsigma_3(\cdot)$ is   determined only by $f$, $Q$, $T$ and has the zero limit at $0$.

\end{lemma}
\begin{proof}  Plugging the estimate proved in Lemma~\ref{lm:extremal_shift_two_point} into (\ref{ineq:distance_X_x_star_y_bar_eta_t_k_plus}), we conclude that
	\[
	\begin{split}
	W_2^2(m(t_{k+1}),\widetilde{\eta(t_{k+1})}&)\\\leq (1+(3L+1)&(t_{k+1}-t_k))\int_{\td\times\mathcal{S}}\|x_*-\bar{y}_*\|^2\pi_k(d(x_*,\bar{y}_*))\\+L W_2^2(m&(t_{k}),\widetilde{\mu(t_k)})(t_{k+1}-t_k)+2\varepsilon^2 (t_{k+1}-t_k)+\varsigma_2(t_{k+1}-t_k)\cdot(t_{k+1}-t_k)\\ = (1+(3L+1)&(t_{k+1}-t_k))W_2^2(m(t_{k}),\widetilde{\eta(t_{k})})\\+L W_2^2(m&(t_{k}),\widetilde{\mu(t_k)})(t_{k+1}-t_k)+2\varepsilon^2 (t_{k+1}-t_k)+\varsigma_2(t_{k+1}-t_k)\cdot(t_{k+1}-t_k).
	\end{split}\] Here we used the fact that $\pi_k$ is an optimal plan between $m(t_{k})$ and $\widetilde{\eta(t_{k})}$ (see~\ref{intro:varsigma_2}). Then,  notice that $\varsigma_2(t_{k+1}-t_k)\leq \varsigma_2(d(\Delta))$.
	 Further, due to Proposition~\ref{prop:metrics} and Lemma~\ref{lm:mu_eta_estiame} we have that
	 \[\begin{split}
	 W_2^2&(m(t_{k}),\widetilde{\mu(t_{k})})\\ &\leq W_2^2(m(t_{k}),\widetilde{\eta(t_{k})})+W_2^2(\widetilde{\mu(t_{k})},\widetilde{\eta(t_{k})})+2W_2(m(t_{k}),\widetilde{\eta(t_{k})})\cdot W_2^2(\widetilde{\mu(t_{k})},\widetilde{\eta(t_{k})})\\ &\leq 
	 W_2^2(m(t_{k}),\widetilde{\eta(t_{k})})+C_2C_5d(\Delta)+2\sqrt{dC_2C_5}\cdot\sqrt{d(\Delta)}.
	 \end{split}\] Hence we obtain the conclusion of the Lemma with
	\[\varsigma_3(\theta)\triangleq \varsigma_2(\theta)+L\left(C_2C_5\cdot\theta+2\sqrt{dC_2C_5\cdot \theta}\right).\]
\end{proof}

\begin{proof}[Proof of Theorem~\ref{th:existence_strategy}]
Given $m(\cdot)\in \mathcal{X}_1(t_0,m_0,\Delta,\mathfrak{u}^\Delta)$, we obtain the auxiliary motions of the Markov chain $\mu(\cdot)$ and $\eta(\cdot)$. Notice that $\mu(\cdot)$ satisfies
\[\varphi^+(t_{k+1},\mu(t_{k+1}))\leq \varphi^+(t_k,\mu(t_k)).\] Thus, \begin{equation}\label{ineq:varphi_mu}
\varphi^+(T,\mu(T))\leq \varphi^+(t_0,\mu(t_0)).
\end{equation} Applying Lemma~\ref{lm:extremal_shift_m_mu} sequentially, we obtain that
\begin{equation*}\label{ineq:m_eta_0_T}
W_2^2(m(T),\widetilde{\eta(T)})\leq e^{(4L+1)(T-t_0)}[W_2^2(m(t_0),\widetilde{\eta(t_0)})+2\varepsilon^2(T-t_0)+\varsigma_3(d(\Delta))\cdot(T-t_0)].
\end{equation*}
Recall that $\mu(t_0)=\eta(t_0)=\operatorname{pr}_{\mathcal{S}}(m_0)$ (see~\ref{strategy_intro:initial}). Therefore taking into account the fact that $\widetilde{\operatorname{pr}_{\mathcal{S}}(m_0)}$ is the nearest to $m_0$ element of $\mathcal{P}^2(\mathcal{S})$ and assumption (\ref{assumption:ineq:S_td}), we have that
\[W_2^2(m(t_0),\widetilde{\eta(t_0)})\leq \varepsilon^2.\] Therefore,
\begin{equation}\label{ineq:distance_m_eta_T}
\begin{split}
W_2(m(T),\widetilde{\eta(T)})&\leq e^{(2L+1/2)T}[(1+2T)\varepsilon^2+T\varsigma_3(d(\Delta))]^{1/2}\\ &\leq \sqrt{1+2T}e^{(2L+1/2)T}\varepsilon+\sqrt{T\varsigma_3(d(\Delta))}.
\end{split}
\end{equation}

Lemma~\ref{lm:mu_eta_estiame} and Proposition~\ref{prop:metrics} give that 
\[W_2(\widetilde{\eta(T)},\widetilde{\mu(T)})\leq \sqrt{C_2C_5d(\Delta)}.\]

Combining this with (\ref{ineq:distance_m_eta_T}), we arrive at the estimate
\begin{equation}\label{ineq:distance_m_mu_final}
W_2(m(T),\widetilde{\mu(T)})\leq C^*\varepsilon+\varsigma^*(d(\Delta)).
\end{equation} Here, we use the designation introduced in Section~\ref{sect:result}: \[C^*\triangleq\sqrt{1+2T}e^{(2L+1/2)T},\] and additionally put
\[\varsigma^*(\theta)\triangleq \sqrt{T\varsigma_3(\theta)}+\sqrt{C_2C_5\theta}.\] Notice that $\varsigma^*(\theta)\rightarrow 0$ as $\theta\rightarrow 0$. Recall that $\varsigma_g$ is a modulus of continuity of the function~$g$. Estimate (\ref{ineq:distance_m_mu_final}) implies that
\[g(m(T))\leq g(\widetilde{\mu(T)})+\varsigma_g(C^*\varepsilon+\varsigma^*(d(\Delta))).\] To complete the proof of the first part of the theorem it suffices to apply inequality (\ref{ineq:varphi_mu}) and condition~\ref{def:part:supersolution_boundary} that states that
\[g(\widetilde{\mu(T)})\leq \varphi^+(T,\mu(T)).\]

As it was mentioned above, the second part of Theorem~\ref{th:existence_strategy} can be proved by interchanging of the players.

\end{proof}


\section*{Appendix. A property of Wasserstein metric on a simplex}
The purpose of the Appendix is to prove the inequalities between the Wasserstein metric on $\mathcal{P}(\mathcal{S})$ and the $p$-th metric on $\Sigma$ formulated in Proposition~\ref{prop:metrics}. 

\begin{proof}[Proof of Proposition~\ref{prop:metrics}]
Since the set $\mathcal{S}$ is finite we have  that, for  $\pi\in \Pi(\widetilde{\mu^1},\widetilde{\mu^2})$,  there exist nonnegative numbers $b_{\bar{x},\bar{y}}[\pi]$, $\bar{x},\bar{y}\in\mathcal{S}$, such that	
\begin{equation}\label{equality:pi_S_delta}
\pi=\sum_{\bar{x},\bar{y}\in\mathcal{S}}b_{\bar{x},\bar{y}}[\pi]\delta_{\bar{x},\bar{y}}. 
\end{equation}
	
Now let us prove inequality (\ref{ineq:R_d_wasser}). Let $\pi_0\in\Pi(\widetilde{\mu^1},\widetilde{\mu^2})$ be such that
$$W_p^p(\widetilde{\mu^1},\widetilde{\mu^2})=\int_{\mathcal{S}\times\mathcal{S}}\|x-y\|^p\pi_0(d(x,y)). $$  From (\ref{equality:pi_S_delta}) it follows that  
$$W_p^p(\widetilde{\mu^1},\widetilde{\mu^2})=\sum_{\bar{x},\bar{y}\in\mathcal{S}}b_{\bar{x},\bar{y}}[\pi_0]\|\bar{x}-\bar{y}\|^p. $$ Notice that $$\sum_{\bar{y}\in\mathcal{S},\bar{y}\neq \bar{x}}b_{\bar{x},\bar{y}}[\pi_0]\geq (\mu_{\bar{x}}^2-\mu_{\bar{x}}^1)^+, $$ where $a^+$ is equal to $a$ provided that $a$ is positive, and $0$ otherwise. Using this estimate and the definition of the fineness of the lattice $\mathcal{S}$ (see (\ref{intro:fineness_S})), we conclude that
$$W_p^p(\widetilde{\mu^1},\widetilde{\mu^2})\geq \sum_{\bar{x}\in\mathcal{S}}\sum_{\bar{y}\in\mathcal{S}}b_{\bar{x},\bar{y}}[\pi_0](d(\mathcal{S}))^p\geq (d(\mathcal{S}))^p\sum_{\bar{x}\in \mathcal{S}}(\mu_{\bar{x}}^2-\mu_{\bar{x}}^1)^+.  $$ Since $\mu^1,\mu^2\in\Sigma$, we have that
\begin{equation}\label{equality:mu_mu_plus}
\sum_{\bar{x}\in \mathcal{S}}(\mu_{\bar{x}}^2-\mu_{\bar{x}}^1)^+=\frac{1}{2}\sum_{\bar{x}\in \mathcal{S}}|\mu_{\bar{x}}^2-\mu_{\bar{x}}^1|. 
\end{equation} Furthermore, for every $\bar{x}\in\mathcal{S}$, $|\mu_{\bar{x}}^2-\mu_{\bar{x}}^1|\leq 1$. Thus, we have that
$$W_p^p(\widetilde{\mu^1},\widetilde{\mu^2})\geq \frac{(d(\mathcal{S}))^p}{2}\sum_{\bar{x}\in \mathcal{S}}|\mu_{\bar{x}}^2-\mu_{\bar{x}}^1|\geq \frac{(d(\mathcal{S}))^p}{2}\sum_{\bar{x}\in \mathcal{S}}|\mu_{\bar{x}}^2-\mu_{\bar{x}}^1|^p. $$ This proves (\ref{ineq:R_d_wasser}).

To prove inequality (\ref{ineq:wasser_R_d}), choose a probability $\pi'\in\Pi(\widetilde{\mu^1},\widetilde{\mu^2})$ such that $$b_{\bar{x},\bar{x}}[\pi']\triangleq\mu^1_{\bar{x}}\wedge\mu_{\bar{x}}^2. $$ Here $b_{\bar{x},\bar{y}}[\pi']$ are given by representation (\ref{equality:pi_S_delta}). Therefore,
\begin{equation}\label{equality:sum_b_mu}
\sum_{\bar{y}\in\mathcal{S},\bar{y}\neq\bar{x}}b_{\bar{x},\bar{y}}[\pi']=(\mu_{\bar{x}}^2-\mu_{\bar{x}}^2)^+.
\end{equation}

The inclusion $\mathcal{S}\subset\td$ yields 
$$W_p^p(\widetilde{\mu^1},\widetilde{\mu^2})\leq\sum_{\bar{x},\bar{y}\in\mathcal{S}} \|\bar{x}-\bar{y}\|^pb_{\bar{x},\bar{y}}[\pi']\leq d^{p/2}\sum_{\bar{x}\in \mathcal{S}}\sum_{\bar{y}\in\mathcal{S},\bar{y}\neq \bar{x}}b_{\bar{x},\bar{y}}[\pi']. $$ This and equalities (\ref{equality:mu_mu_plus}), (\ref{equality:sum_b_mu}) imply that
$$W_p^p(\widetilde{\mu^1},\widetilde{\mu^2})\leq \frac{d^{p/2}}{2}\sum_{\bar{x}\in \mathcal{S}} |\mu_{\bar{x}}^2-\mu_{\bar{x}}^1|=\frac{d^{p/2}}{2}\|\mu^2-\mu^1\|_1. $$
The Holder's inequality gives that
$$W_p^p(\widetilde{\mu^1},\widetilde{\mu^2})\leq \frac{d^{p/2+p'}}{2}\|\mu^2-\mu^1\|_p, $$ where $p'$ is such that $1/p'+1/p=1$. This proves (\ref{ineq:wasser_R_d})
\end{proof}

\bibliography{mfdg}

\end{document}